\documentclass[reqno]{amsart}
\usepackage{amssymb, latexsym, amsmath, amsfonts, amscd}
\usepackage{lscape,color}
\usepackage{mathrsfs, hyperref}

\newtheorem{lemma}{Lemma}
\newtheorem{theorem}{Theorem}
\newtheorem{proposition}{Proposition}

{\theoremstyle{definition}}
{\theoremstyle{definition}}
{\theoremstyle{definition}}
{\theoremstyle{definition}}

\numberwithin{lemma}{section}

\newcommand{\R}{\mathbf{R}}

\newcommand{\N}{\mathbf{N}}

\newcommand{\sech}{\textnormal{sech}}
\newcommand{\p}{\partial}

\newcommand{\supp}{\textnormal{supp }}

\title{Asymptotic Stability for KdV Solitons in Weighted Spaces via Iteration}
\author[B. Pigott]{Brian Pigott}
\address{Wake Forest University\\pigottbj@wfu.edu}
\author[S. Raynor]{Sarah Raynor}
\address{Wake Forest University\\raynorsg@wfu.edu}
\date{\today}
\thanks{S. Raynor would like to thank the Simons Foundation for their support during the dreation of this work.  The authors would like to thank Jeremy Marzuola for his helpful suggestions on a draft of this work.}

\begin{document}

\begin{abstract}
In this paper, we reconsider the well-known result of Pego-Weinstein \cite{MR1289328} that soliton solutions to the Korteweg-deVries equation are asymptotically stable in exponentially weighted spaces.  In this work, we recreate this result in the setting of modern well-posedness function spaces.  We obtain asymptotic stability in the exponentially weighted space via an iteration argument.  Our purpose here is to lay the groundwork to use the $I$-method to obtain asymptotic stability below $H^1$, which will be done in a second, forthcoming paper \cite{PR}.  This will be possible because the exponential approach rate obtained here will defeat the polynomial loss in traditional applications of the $I$-method \cite{MR1995945}, \cite{MR1951312}, \cite{pigottorb}.
\end{abstract}

\maketitle

\section{Introduction}
We consider solutions to the Korteweg-de Vries equation:
$$
u_t+u_{xxx}+\p_x(u^2)=0,
$$
which is a well-known nonlinear dispersive partial differential equation modelling the behavior of water waves in a long, narrow, shallow canal.  Of particular interest are soliton solutions to this equation, which are special travelling wave solutions of the form
\begin{equation}\label{soliton}
Q_{c,x_0}(x,t)=\psi_c(x-ct-x_0)=\frac{3c}{2}\sech^2(\frac{\sqrt{c}}{2}(x-ct-x_0)).
\end{equation}
The stability of these solitons has been an area of intense study for many years.  One might first be interested in the orbital stability of the soliton.  That is, if, at $t=0$, $u(x,0)-\psi_c(x)$ is small in an appropriate norm, then, for all time there is some $x_0(t)$ so that \\ $u(x,t)-\psi_c(x-x_0(t))$ remains small.  The study of orbital stability in the energy space $H^1$ began with with Benjamin \cite{MR0338584} and Bona \cite{MR0386438}, continuing with Weinstein \cite{MR820338}.  Merle and Vega established the orbital and asymptotic stability of KdV solitons in $L^2$ \cite{MR1949297}.  One can also study the possibility of orbital stability of solitons in $H^s$ for $s$ not an integer, and in \cite{MR1995945}, \cite{pigottorb} it was shown that, for $0<s<1$, the possible orbital instability of the solitons is at most polynomial in time.

Also of interest is the concept of asymptotic stability, meaning that there exist $c_+$ and $x_+$ so that, in some appropriate sense, $u(x,t)-\psi_{c_+}(x-c_+t-x_+)$ goes to zero as time goes to positive infinity.  Asymptotic stability for the Korteweg-deVries equation was first studied by Pego and Weinstein in \cite{MR1289328}.  In that paper, the authors considered the behavior of solutions to KdV in the weighted space $H^1_a=\{f|\|e^{ax}f(x)\|_{H^1} < + \infty\}$, for appropriate choice of $a$.  In that setting, they were able to conclude that solitons are asymptotically stable and, in fact, converge exponentially to the limiting soliton.   Asymptotic stability in the space $H^1$ was established by Martel and Merle in \cite{MR2109467,MR1826966}, and in $L^2$ by Merle and Vega via the Miura transform \cite{MR1949297}.

In this paper, we reconsider the result of Pego and Weinstein.  We establish local well-posedness for the exponentially weighted soliton perturbation in a space $X^{1,1/2,1}$ which embeds into the Bourgain space $X^1_b$, partially following the local well-posedness work of Molinet and Ribaud \cite{MR1889080} on dispersive-dissipative equations; see also \cite{MR1918236}, \cite{MR2514729}.  We then run an iteration scheme to establish global control of the perturbation in $H^1$ and the exponentially weighted space $H^1_a$, concluding that the soliton is exponentially asymptotically stable in $H^1_a$.  The purpose of this work is to modernize the techniques used in Pego-Weinstein so that the result can be used in concert with other modern techniques.  In particular, in a forthcoming work we will combine this argument with the $I$-method to obtain asymptotic stability of solitons in the space $H^s_a$.

In previous work of Colliander et al.\cite{MR1951312}, Raynor and Staffilani \cite{MR1995945}, and Pigott \cite{pigottorb}, the question of orbital stability of solitons in $H^s$ for $0<s<1$ has been considered, for the nonlinear Schr\"odinger, KdV, and generalized KdV equations respectively.  In each case, it was found that there was at most a polynomial instability in the orbit.  That is, if $\|u_0-\psi_c\|_{H^s}$ is sufficiently small, then for some long time interval and some power $p$ depending only on $s$, $\inf_{x_0\in\R}\|u(t,\cdot)-\psi_c(\cdot-x_0)\|_{H^s} \lesssim Ct^p$.  This possibility of polynomial growth in the $H^s$ error is believed to be an artifact of the technique; whenever estimates are done with the $I$-method, there is some small error on each time step which grows polynomially when iterated.  Our purpose in these papers is to obviate that error.  When working in the Pego-Weinstein weighted spaces, one gains an exponential decay in the error as the solution evolves.  Therefore, one can hope to defeat the polynomial error and obtain a true stability result in $H^s$ for $s<1$, which will be a substantial improvement in the current state of the art.  However, working with both the dissipative spectral structure of the weighted spaces and the delicate harmonic analysis required for the $I$-method turns out to be rather technical.  For that reason, we have split the work in two.  In this paper, we update the well-posedness technology necessary to use both methods, and then demonstrate how this can be utilized by reproducing the well-known $H^1$ stability result.  The result below $H^1$ will appear in a subsequent paper \cite{PR}.

The paper is organized as follows: In section 2, we will set up our notation and establish basic results.  In section 3, we will establish the necessary local well-posedness.  In section 4, we will run the iteration scheme and establish the main result of the paper.

\section{Notation and Basic Results}

Consider a solution $u(x,t)$ to the Cauchy problem:
\begin{align}
u_t+u_{xxx}+\p_x(u^2) & = 0& & \forall x \in \R \ \ \forall t > 0 \label{KdV}\\
u(x,0)& = u_0(x) &  & \forall x \in \R. \notag
\end{align}
Let $c > 0$ and consider the function $\psi_c$ which is the unique even, exponentially decaying solution to the soliton equation
$$-c\psi_c'+\psi_c'''+(\psi^2)'=0 .$$ The function $\psi_c$ takes the form given in \eqref{soliton}.  In this work we will consider initial conditions $u(x,0)=\psi_{c_0}(x)+ v_0(x)$, where $\|v_0\|_{H^1}$ is sufficiently small.  We will make the ansatz that $u(x,t)=\psi_{c(t)}(x-\int_0^tc(s)ds-\gamma(t)) + v(x-\int_0^tc(s)ds-\gamma(t),t)$ where $c(t)$ and $\gamma(t)$ will be chosen later.  From now on, we denote by $y$ the quantity $x-\int_0^tc(s)ds-\gamma(t)$.

We are also interested in controlling $v$ in the space $H^1_a,$ which has the norm \\$\|f\|_{H^1_a} = \|e^{ay}f\|_{H^1}.$  This is equivalent to controlling $w(y,t) :=e^{ay}v(y,t)$ in $H^1$.  We will choose the parameters $c(t)$ and $\gamma(t)$ so that $\|w(\cdot,t)\|_{L^2}$ is minimized at each $t$.  In order to do so, we first need to consider the difference equations satisfied by $v$ and $w$, and consider their linearizations about the soliton.
\begin{lemma} The perturbation $v$ satisfies the difference equation
\begin{equation}\label{vequation}
v_t=\p_y(-\p_y^2+c_0-2\psi_c)v + \p_y(v^2) + (\dot{\gamma}\p_y + \dot{c}\p_c)\psi_c + (\dot{\gamma}+c-c_0)\p_yv
\end{equation}
Moreover, the perturbation $w$ satisfies the difference equation
\begin{equation}\label{wequation}
\begin{aligned}
w_t &= e^{ay}\p_y(-\p_y^2+c_0-2\psi_c)e^{-ay}w + (c-c_0)(\p_y-a)w \\
&\qquad + [e^{ay}(\dot{c}\p_c + \dot{\gamma}\p_y]\psi_c + \dot{\gamma}(\p_y-a)w + e^{ay}\p_y(c-c_0 + v^2)e^{-ay}w].
\end{aligned}
\end{equation}
\end{lemma}
\begin{proof} For these calculations, see \cite{MR1289328}.
\end{proof}

For fixed $c>0$, define the operator $A_a=e^{ay}\p_y(-\p_y^2+c-2\psi_c)e^{-ay}.$ We have the following from \cite{MR1289328}:
\begin{proposition}\label{spectral} For $0 < a < \sqrt{\frac{c}{3}},$ the spectrum of $A_a$ in $L^2$ consists of the following:
\begin{enumerate}
\item An eigenvalue of algebraic multiplicity $2$ at $\lambda=0$.  A generator of the kernel of $A_a$ is $\zeta_1=e^{ay}\p_y\psi_c$, and the second generator of the generalized kernel of $A_a$ is $\zeta_2=e^{ay}\p_c\psi_c$.
\item A continuous spectrum $S^a$ parametrized by $\tau \to i\tau^3-3a\tau^2+(c-3a^2)i\tau-a(c-a^2)$.  For any element $\lambda$ of this continuous spectrum, the real part of $\lambda$ is at most $-a(c-a^2)<0$.
\end{enumerate}
The spectrum contains no other elements.
\end{proposition}
We also need to consider the elements of the spectrum to $A_a^*$, which are \\$\eta_1=e^{-ay}[\theta_1\p_y^{-1}\p_c\psi_c + \theta_2\psi_c]$ and $\eta_2=e^{-ay}(\theta_3\psi_c)$, where $\p_y^{-1}f$ is defined to be $\int_{-\infty}^yf(t)dt$ and $\theta_1$, $\theta_2$ and $\theta_3$ are appropriate constants to obtain the biorthogonality relationship $\langle \zeta_j,\eta_k\rangle=\delta_{jk}$.  We will define the $L^2$ spectral projections $Pw=\sum_{i=1}^2\langle w, \eta_i\rangle \zeta_i$ and $Qw=w-Pw$ onto the discrete and continuous spectrums of $A_a$ respectively, with respect to the fixed initial value of $c$, $c_0$.  Finally, note that the spectrum is the same in $H^1$:
\begin{proposition}\label{H1spec} For $0 < a < \sqrt{\frac{c}{3}},$ the spectrum of $A_a$ in $H^1$ is the same as its spectrum in $L^2$.
\end{proposition}
\begin{proof} First, note that the kernel elements listed above are also elements of $H^1$.  Therefore the discrete spectrum is the same in $H^1$.  (Since $H^1 \subset L^2$.)  Second, note that \\$A_a=e^{ay}\p_y(-\p_y^2+c-2\psi_c)e^{-ay}=(\p_y-a)(-(\p_y-a)^2+c_0)-2(\p_y-a)\psi_c.$  Because $\psi_c$ is an exponentially decaying $C^\infty$ function, the operator $-2(\p_y-a)\psi_c$ is a compact perturbation of $A_a^0:=(\p_y-a)(-(\p_y-a)^2+c_0),$ whose continuous spectrum in $H^1$ is exactly as in $L^2$.
\end{proof}

Returning to the difference equation \eqref{wequation}, for each fixed $t$ we select $c(t)$ and $\gamma(t)$ so that $Pw=0$, and $Qw=w$.
Defining ${\mathcal{F}}=[e^{ay}(\dot{c}\p_c + \dot{\gamma}\p_y]\psi_c + \dot{\gamma}(\p_y-a)w + e^{ay}\p_y(c-c_0 + v^2)e^{-ay}w],$ we have that
$$w_t = A_aw + Q{\mathcal{F}},$$ and
\begin{eqnarray}\label{modulation}\begin{bmatrix} 1 + \langle e^{ay}(\p_y\psi_c-\p_y\psi_{c_0}),\eta_1 \rangle -\langle w, \p_y\eta_1 \rangle & \langle e^{ay}(\p_c\psi_c-\p_c\psi_{c_0}),\eta_1 \rangle \\ \langle  e^{ay}(\p_y\psi_c-\p_y\psi_{c_0}),\eta_2 \rangle-\langle w, \p_y\eta_2 \rangle & 1 + \langle e^{ay}(\p_c\psi_c-\p_c\psi_{c_0}),\eta_2 \rangle \end{bmatrix}\begin{bmatrix} \dot{\gamma} \\ \dot{c} \end{bmatrix} \phantom{blahblahblahblahblah} \\ \phantom{blahblahblahblahblah} = \begin{bmatrix} \langle e^{ay}\p_y(c-c_0 + v^2)e^{-ay}w,\eta_1 \rangle \\ \langle e^{ay}\p_y(c-c_0 + v^2)e^{-ay}w, \eta_2 \rangle \end{bmatrix}\end{eqnarray}

\section{Local-in-Time Theory}
The modulation equation for $w$ can be written out as:

\begin{equation}
\label{wequationmod}
\left \{ \begin{array}{l l}
\partial_{t} w + \partial_{x}^{3}w - 3a \partial_{x}^{2} w + (c_{0} - 3a^{2}) \partial_{x}w + a (c_{0} - a^{2}) w - 2 (\partial_{x} - a) (\psi_{c}w) &  \\
\qquad - Q \big [ e^{ax} (\dot{c} \partial_{c} + \dot{\gamma} \partial_{x}) \psi_{c} + \dot{\gamma} (\partial_{x} - a) w + e^{ax} \partial_{x}(c - c_{0} + v) e^{-ax} w \big ] = 0, & \\
w(0,x) = w_{0}(x). &
\end{array}
\right .
\end{equation}

To implement our iteration argument, we need to establish control of the perturbations $v$ and $w$ as follows: suppose that $v_{0}, w_{0} \in H^{1}(\R)$ and let $v(t,x), w(t,x)$ satisfy \eqref{vequation} and \eqref{wequationmod}, respectively. There is a $\delta > 0$ such that
\begin{equation}
\label{targetcontrol}
\| v \|_{X^{1}_{\delta}} \lesssim \| v_{0} \|_{H^{1}} \qquad \text{and} \qquad \| w \|_{X^{1}_{\delta}} \lesssim \| w_{0} \|_{H^{1}}
\end{equation}
in some space $X^{1}_{\delta}$ of space-time functions localized to the time interval $[0,\delta]$; see Proposition \ref{vwlwp} below.

It turns out that the selection of the space $X^{1}_{\delta}$ is a rather delicate matter, owing in large part to the requirement that it must accommodate solutions of both \eqref{vequation} and \eqref{wequation}. A natural candidate is the space $X^{1,b}$ with $b > 1/2$ defined by
\begin{equation*}
\| f \|_{X^{1,b}} := \left \| \langle \tau - \xi^{3} \rangle^{b} \langle \xi \rangle \widetilde{f} \right \|_{L^{2}_{\tau,\xi}},
\end{equation*}
where $f(t,x)$ is a space-time function and $\widetilde{f}(\tau,\xi)$ is its (space-time) Fourier transform. These spaces were successfully implemented in the study of the KdV equation (see \cite{MR1215780}, \cite{MR1329387}) and would be sufficient to establish \eqref{targetcontrol} for $v$. A theory applicable to the weighted perturbation $w$ is also available in $X^{1}_{b}$ (see \cite{MR1889080}), however, Molinet and Ribaud prove that
\begin{equation*}
\| w \|_{X^{1,b}} \lesssim \| w_{0} \|_{H^{1 + (2b-1)}}.
\end{equation*}
Since we require $b > 1/2$ to obtain the embedding $X^{s,b} \hookrightarrow C_{t}H^{s}_{x}$, this estimate is insufficient for our purposes.  Alternatively, one could try to accommodate the presence of the dissipative term in the definition of the $X^{1,b}$ space; for instance one could consider the space $Y^{1,b}$ with norm
\begin{equation*}
\| f \|_{Y^{1,b}} := \left \| \langle i(\tau - \xi^{3}) + \xi^{2} \rangle^{b} \langle \xi \rangle \widetilde{f} \right \|_{L^{2}_{\tau,\xi}}
\end{equation*}
In this case there is an adequate theory to handle the equation for the weighted perturbation $w$ (see \cite{MR1918236}); however this space is no longer suitable for the unweighted perturbation $v$.

Following this reasoning, we are lead to consider a Besov refinement of the space $X^{1,b}$, which enables us to choose $b = 1/2$ while still having an embedding into $C^{0}_{t} H^{1}_{x}$. We begin with some notation. Define the sets $A_{j}$ and $B_{k}$ by
\begin{align*}
A_{j} &:= \{ (\tau, \xi) \in \R^{2} \ \vert \ 2^{j} \leq \langle \xi \rangle \leq 2^{j+1} \}, \qquad j \geq 0,\\
B_{k} &:= \{ (\tau, \xi) \in \R^{2} \ \vert \ 2^{k} \leq \langle \tau - \xi^{3} \rangle \leq 2^{k+1} \}, \qquad k \geq 0.
\end{align*}
Here we regard $\tau,\xi$ as being the frequency variables associated to the time and space variables $t,x$, respectively. We define the space $X^{s,b,1}$ to be the completion of the Schwartz class functions in the norm
\begin{equation}
\label{Xsb1defn}
\| f \|_{X^{s,b,1}} := \left ( \sum_{j \geq 0} 2^{2sj} \left ( \sum_{k \geq 0} 2^{bk}  \| \widetilde{f} \|_{L^{2}(A_{j} \cap B_{k})} \right )^{2} \right )^{1/2}.
\end{equation}
A slightly modified version of this space was used by Kishimoto \cite{MR2501679} in the case where $s = -3/4$ to establish global well-posedness of the KdV equation in $H^{-3/4}(\R)$. Here we are interested mostly in the spaces $X^{s,1/2,1}$ and $X^{s,-1/2,1}$. We recall the following two embeddings valid for $b > 1/2$,
\begin{equation*}
X^{s,b} \hookrightarrow X^{s,1/2,1} \hookrightarrow C^{0}_{t} H^{s}_{x},
\end{equation*}
both of which are easily verified using the Cauchy-Schwartz inequality.

Following standard arguments in the $X^{s,b}$ spaces \cite{MR1209299,MR1215780,MR1329387} we define the time-localized space $X^{s,1/2,1}_{\delta}$ to be the space with norm
\begin{equation*}
\| u \|_{X^{s,1/2,1}_{\delta}} := \inf \{ \| w \|_{X^{s,1/2,1}} \ \vert \ w \equiv u \ \text{on} \ [0,\delta] \}.
\end{equation*}

Since we work primarily in frequency space, we define the space $\widetilde{X}^{s,1/2,1}$ corresponding to the norm
\begin{equation}
\| f \|_{\widetilde{X}^{s,1/2,1}} := \left ( \sum_{j \geq 0} 2^{2sj} \left ( \sum_{k \geq 0} 2^{bk}  \| f \|_{L^{2}(A_{j} \cap B_{k})} \right )^{2} \right )^{1/2},\label{Xs,1/2,1}
\end{equation}
where $f = f(\tau,\xi)$.

\subsection{Linear estimates}
We introduce notation for the linear evolutions corresponding to \eqref{vequation} and \eqref{wequation}. Let $W_{1}(t)$ denote the standard Airy evolution,
\begin{equation*}
\mathcal{F}_{x} \left ( W_{1}(t)f \right )(\xi) = e^{-it\xi^{3}} \widehat{f}(\xi).
\end{equation*}
We also introduce $W_{2}(t)$, which we define for $t \geq 0$ by
\begin{equation*}
\mathcal{F}_{x} \left ( W_{2}(t)f \right )(\xi) = e^{-it\xi^{3} - p_{a}(\xi)t} \widehat{f}(\xi),
\end{equation*}
where $p_{a}(\xi) = 3a\xi^{2} + a(c_{0}^{2} - a)$.
We extend this to all of $t \in \R$ by
\begin{equation*}
\mathcal{F}_{x} \left ( W_{2}(t) f\right )(\xi) = e^{-it\xi^{3} - p_{a}(\xi) \vert t \vert} \widehat{f}(\xi).
\end{equation*}

In what follows we let $\rho$ be a time cut off function such that
\begin{equation*}
\rho \in C_{0}^{\infty}(\R),\qquad \supp \ \rho \subset [-2,2], \qquad \rho \equiv 1 \ \text{on} \ [-1, 1].
\end{equation*}
Set $\rho_{T}(\cdot) = \rho(\cdot/T)$.

\begin{lemma}
\label{linearairy}
With $W_{1}(t)$ defined as above we have the following two linear estimates:
\begin{align}
\| \rho(t) W_{1}(t)f \|_{X^{s,1/2,1}} &\lesssim \| f \|_{H^{s}}, \label{linear1}\\
\left \| \rho(t) \int_{0}^{t} W_{1}(t-s) F(s) ds \right \|_{X^{s,1/2,1}} &\lesssim \| F \|_{X^{s,-1/2,1}}. \label{linear2}
\end{align}
\end{lemma}

\begin{proof}
Recall that $\| W_{1}(t) f \|_{X^{s,1/2+}} \lesssim \| f \|_{H^{s}}$. By a similar argument, using the construction of the space and H\"older's inequality, we obtain \eqref{linear1}. The estimate \eqref{linear2} is established in Lemma 4.1 of \cite{MR2501679}.
\end{proof}

\begin{lemma}
\label{lineardissipative}
Let $0 < a \leq \min(1,c_{0})$, and let $s \in \R$. For $W_{2}(t)$ defined as above we have the following two linear estimates:
\begin{align}
\| \rho(t) W_{2}(t)f \|_{X^{s,1/2,1}} &\lesssim \| f \|_{H^{s}}, \label{linear3}\\
\left \| \chi_{\R_{+}}(t) \rho(t) \int_{0}^{t} W_{2}(t-s) F(s) ds \right \|_{X^{s,1/2,1}} &\lesssim \| F \|_{X^{s,-1/2,1}}. \label{linear4}
\end{align}
\end{lemma}

\begin{proof}
Our proof of \eqref{linear3} mimics the argument given in \cite{MR2514729} Proposition 4.3. Observe that it suffices to prove that for each $j \geq 0$,
\begin{equation*}
 \| \rho(t) W_{2}(t) \widehat{f} \|_{L^{2}(A_{j})} \lesssim \| \widehat{f} \|_{L^{2}(A_{j})}.
\end{equation*}
If $j=0$, then $\vert p_{a}(\xi) \vert \leq 3a + a(c_{0}^{2} - a)$, and
\begin{equation}
\label{homeq1}
\begin{aligned}
\| \rho(t) W_{2}(t) \widehat{f}_{0} \|_{L^{2}}
&\lesssim \sum_{k \geq 0} 2^{k/2} \| \mathcal{F}_{t} \left ( \rho(t) e^{-p_{a}(\xi)\vert t \vert} \right ) \widehat{f}_{0} \|_{L^{2}(B_{k})}\\
&\lesssim \sum_{n \geq 0} \frac{(3a + a(c_{0}^{2}-a))^{n}}{n!} \| \widehat{f}_{0} \|_{L^{2}} \| \rho(t) \vert t \vert^{n} \|_{H^{1}_{t}} \lesssim \| f_{0} \|_{L^{2}}.
\end{aligned}
\end{equation}

Let $P_{k}$ be the projection operator defined by
\begin{equation*}
\mathcal{F} (P_{k} \psi)(\tau,\xi) = \chi_{B_{k}}(\tau,\xi) \mathcal{F}(\psi)(\tau,\xi).
\end{equation*}
%{\color{red}{Should this be $\chi_{B_{k}}(\tau,\xi) \widetilde{\psi}(\tau,\xi)$.}}
Notice that if $\xi \sim 2^{j}$, then for any $k \geq 0$ we have
\begin{equation*}
\| P_{k}(\exp(-p_{a}(\xi) \vert t \vert))(t) \|_{L^{2}_{t}} \lesssim \| P_{k}(\exp(-p_{a}(2^{j}) \vert t \vert))(t) \|_{L^{2}_{t}}.
\end{equation*}
This follows from Plancherel and the fact that
\begin{equation*}
\mathcal{F}_{t}(e^{-\vert t \vert})(\tau) = C \frac{1}{1 + \vert \tau \vert^{2}}.
\end{equation*}

In the case when $j \geq 1$ we have
\begin{align*}
\| \rho(t) W_{2}(t) \widehat{f}_{j} \|_{L^{2}_{\xi}}
&\lesssim \sum_{k \geq 0} 2^{k/2} \left \| \widehat{f}_{j}(\xi) \mathcal{F}_{t} \left ( \rho(t) e^{-p_{a}(\xi)\vert t \vert} \right) \right \|_{L^{2}(B_{k})}\\
&\lesssim \sum_{k \geq 0} 2^{k/2} \| \widehat{f}_{j} \|_{L^{2}_{\xi}} \left \| \chi_{A_{j}}(\xi) P_{k}(\rho(t) \exp(-p_{a}(\xi) \vert t \vert))(t) \right \|_{L^{\infty}_{\xi} L^{2}_{t}}.
\end{align*}
It suffices to show that for each each $j \geq 1$ the sum
\begin{equation}
\label{disslinear1}
\sum_{k \geq 0} 2^{k/2} \left \| \chi_{A_{j}}(\xi) P_{k}(\rho(t) \exp(-p_{a}(\xi) \vert t \vert))(t) \right \|_{L^{\infty}_{\xi} L^{2}_{t}},
\end{equation}
is bounded.
We may assume that $k \geq 50$ in the summation. We first decompose the product $\rho(t) \exp(-p_{a}(\xi) \vert t \vert )$ using the following para-product decomposition:
\begin{equation*}
\rho \phi = \sum_{i \geq 0} \big ( (P_{i+1} \rho)(P_{\leq i+1}\phi) + (P_{\leq i}\rho)(P_{i+1} \phi) \big ),
\end{equation*}
where $\rho = \rho(t)$ and $\phi = \exp(-p_{a}(\xi) \vert t \vert)$. We have adopted the notation $P_{\leq k} := \sum_{\ell=0}^{k} P_{\ell}$.
Therefore
\begin{equation*}
P_{k}(\rho \phi) = P_{k} \left ( \sum_{i \geq k - 10} \big ( (P_{i+1} \rho)(P_{\leq i+1}\phi) + (P_{\leq i}\rho)(P_{i+1} \phi) \big ) \right ).
\end{equation*}
We are thus reduced to showing that the following two sums are bounded:
\begin{align*}
I &:= \sum_{k \geq 50} 2^{k/2} \sum_{i \geq k - 10} \left \| \chi_{A_{j}} P_{k} \big ( (P_{i+1} \rho) (P_{\leq i+1} \phi) \big ) \right \|_{L^{\infty}_{\xi} L^{2}_{t}}\\
II &:= \sum_{k \geq 50} 2^{k/2} \sum_{i \geq k - 10} \left \| \chi_{A_{j}} P_{k} \big ( (P_{\leq i} \rho) (P_{i+1} \phi) \big ) \right \|_{L^{\infty}_{\xi} L^{2}_{t}}.
\end{align*}
We estimate $II$ as follows:
\begin{align*}
II &\leq \sum_{k \geq 50} 2^{k/2} \sum_{i \geq k - 10} \| \chi_{A_{j}} P_{i+1} \phi \|_{L^{\infty}_{\xi} L^{2}_{t}} \| P_{\leq i} \rho \|_{L^{\infty}_{t}}\\
&\lesssim \sum_{k \geq 50} \sum_{i \geq k - 10} 2^{(k-i)/2} 2^{i/2} \| P_{i+1} \phi \|_{L^{\infty}_{\xi} L^{2}_{t}}\\
&\lesssim \sum_{i \geq 0} 2^{i/2} \| P_{i + 1} ( \exp(-p_{a}(2^{j}) \vert t \vert) ) \|_{L^{2}_{t}} = \| \exp(-p_{a}(2^{j})\vert t \vert ) \|_{\dot{B}^{1/2}_{2,1}}.
\end{align*}
We recall that the Besov space space $\dot{B}^{1/2}_{2,1}$ has the following scaling structure:
\begin{equation*}
\| f(2^{j} \cdot) \|_{\dot{B}^{1/2}_{2,1}} \sim \| f \|_{\dot{B}^{1/2}_{2,1}}.
\end{equation*}
Since $e^{-\vert t \vert} \in \dot{B}^{1/2}_{2,1}$, the desired result follows. One can estimate $I$ in a similar way.

The proof of \eqref{linear4} proceeds as in the proof of Proposition 4.4 from \cite{MR2514729}. Let
\begin{equation*}
L(F)(t,x) := \chi_{\R_{+}}(t) \rho(t) \int_{0}^{t} W_{2}(t-s) F(s) ds,
\end{equation*}
and define $w(t,x) := W_{1}(-t)F(t,x)$.
Observe that
\begin{align*}
L(F)(t,x) &= \chi_{\R_{+}}(t) \rho(t) \int_{\R} e^{ix \xi} \int_{0}^{t} e^{i(t-s)\xi^{3} - p_{a}(\xi)(t-s)} \widehat{F}(s,\xi) ds d\xi\\
&= \chi_{\R_{+}}(t) \rho(t) \int_{\R} e^{ix \xi} e^{it\xi^{3} - p_{a}(\xi)t} \int_{\R}  \widetilde{w}(\tau,\xi) \int_{0}^{t} e^{p_{a}(\xi)s + is\tau} ds d\tau\\
&= W_{1}(t) \chi_{\R_{+}}(t) \rho(t) \int_{\R^{2}} e^{ix \xi}  \widetilde{w}(\tau,\xi) \left ( \frac{e^{it\tau} - e^{-p_{a}(\xi)t}}{i\tau + p_{a}(\xi)} \right ) d\tau d\xi.
\end{align*}
Define
\begin{equation*}
h(t,\xi) := \rho(t) \int_{\R} \frac{e^{it\tau} - e^{p_{a}(\xi)t}}{i\tau + p_{a}(\xi)} \widetilde{w}(\tau,\xi) d\tau.
\end{equation*}
From the definition of the space $X^{s,1/2,1}$ it suffices to prove
\begin{equation}
\label{inhom1}
\sum_{k \geq 0} 2^{k/2} \| \mathcal{F}_{t}(h)(\tau,\xi) \|_{L^{2}(A_{j} \cap B_{k})} \lesssim \sum_{k \geq 0} 2^{-k/2} \| \widetilde{w}(\tau,\xi) \|_{L^{2}(A_{j} \cap B_{k})}.
\end{equation}
We begin by decomposing $h = h_{1} + h_{2} + h_{3} - h_{4}$, where
\begin{align*}
h_{1}(t,\xi) &:= \rho(t) \int_{\vert \tau \vert \leq 1} \frac{e^{it\tau}-1}{i\tau + p_{a}(\xi)} \widetilde{w}(\tau,\xi) d\tau,\\
h_{2}(t,\xi) &:= \rho(t) \int_{\vert \tau \vert \leq 1} \frac{1 - e^{-p_{a}(\xi)t}}{i\tau + p_{a}(\xi)} \widetilde{w}(\tau,\xi) d\tau,\\
h_{3}(t,\xi) &:= \rho(t) \int_{\vert \tau \vert \geq 1} \frac{e^{it\tau}}{i\tau + p_{a}(\xi)} \widetilde{w}(\tau,\xi) d\tau,\\
h_{4}(t,\xi) &:= \rho(t) \int_{\vert \tau \vert \geq 1} \frac{e^{-p_{a}(\xi)t}}{i\tau + p_{a}(\xi)} \widetilde{w}(\tau,\xi) d\tau.
\end{align*}

\noindent{\textbf{Estimate for $h_{1}$.}} Use a Taylor expansion to see that
\begin{align*}
\sum_{k \geq 0} 2^{k/2} \| \chi_{A_{j}} P_{k}(h_{1}) \|_{L^{2}_{\xi,t}} &\lesssim \sum_{k \geq 0} 2^{k/2} \sum_{n \geq 1} \left \| \chi_{A_{j}} P_{k} \left ( \frac{\rho(t) t^{n}}{n!} \right ) \int_{\vert \tau \vert \leq 1} \frac{\tau^{n}}{i\tau + p_{a}(\xi)} \widetilde{w}(\tau,\xi) d\tau \right \|_{L^{2}_{\xi,t}}\\
&\lesssim \sum_{n \geq 1} \left \| \frac{\rho(t) t^{n}}{n!} \right \|_{B^{1/2}_{2,1}} \left \| \int_{\vert \tau \vert \leq 1} \frac{\vert \tau \vert}{\vert i\tau + p_{a}(\xi) \vert} \chi_{A_{j}}(\xi) \widetilde{w}(\tau,\xi) d\tau \right \|_{L^{2}_{\xi}}\\
&\lesssim \left \| \int_{\vert \tau \vert \leq 1} \frac{\vert \tau \vert}{\vert i\tau + p_{a}(\xi) \vert} \chi_{A_{j}}(\xi) \widetilde{w}(\tau,\xi) d\tau \right \|_{L^{2}_{\xi}}.
\end{align*}
Now we apply H\"{o}lder's inequality in $\tau$ to see that
\begin{align*}
& \left \| \int_{\vert \tau \vert \leq 1} \frac{1}{\vert i\tau + p_{a}(\xi) \vert} \chi_{A_{j}} \widetilde{w}(\tau,\xi) d\tau \right \|_{L^{2}_{\xi}}\\
\leq & \sum_{k \geq 0} \left \| \int_{\vert \tau \vert \leq 1} \frac{\vert \tau \vert}{\vert i\tau + p_{a}(\xi) \vert} \chi_{A_{j}} \chi_{B_{k}} \widetilde{w}(\tau,\xi) d\tau \right \|_{L^{2}_{\xi}}\\
\leq & \sum_{k \geq 0} \left \| \frac{\chi_{A_{j}} \chi_{B_{k}}}{\vert i\tau + p_{a}(\xi) \vert} \right \|_{L^{2}_{\xi,\tau}} \| \chi_{A_{j}} \chi_{B_{k}} \widetilde{w}(\tau,\xi) \|_{L^{2}_{\xi,\tau}}\\
\leq & \sum_{k \geq 0} 2^{-k/2} \| \chi_{A_{j}} \chi_{B_{k}} \widetilde{w} \|_{L^{2}_{\xi,\tau}}.
\end{align*}
%{\color{red}{Do we mean that $\vert \tau \vert \leq 1$, or do we really mean to say that $\vert \tau - \xi^{3} \vert \leq 1$. This needs to be clarified.}}

\noindent{\textbf{Estimate for $h_{2}$.}} If $p_{a}(\xi) \leq 1$, then we use a Taylor expansion to see that
\begin{align*}
&\sum_{k \geq 0} 2^{k/2} \| \chi_{A_{j}} P_{k}(h_{2}) \|_{L^{2}_{\xi,t}}\\
\lesssim & \sum_{n \geq 1} \sum_{k \geq 0} 2^{k/2} \left \| \chi_{A_{j}} \frac{\vert p_{a}(\xi) \vert}{n!} P_{k}(\rho(t) \vert t \vert^{n}) \int_{\vert \tau \vert \leq 1} \frac{\widetilde{w}(\tau,\xi)}{i\tau + p_{a}(\xi)} d\tau \right \|_{L^{2}_{\xi,t}}\\
\lesssim & \left \| \int_{\vert \tau \vert \leq 1} \frac{\chi_{A_{j}} \widetilde{w}(\tau,\xi)}{\vert i\tau + p_{a}(\xi) \vert} d\tau \right  \|_{L^{2}_{\xi}}\\
\lesssim & \sum_{k \geq 0} \left \| \frac{\chi_{A_{j}} \chi_{B_{k}}}{i\tau + p_{a}(\xi)} \right \|_{L^{\infty}_{\xi}L^{2}_{\tau}} \| \chi_{A_{j}} \chi_{B_{k}} \widetilde{w} \|_{L^{2}_{\tau,\xi}}\\
\lesssim & \sum_{k \geq 0} 2^{-k/2} \| \chi_{A_{j}} \chi_{B_{k}} \widetilde{w}(\tau,\xi) \|_{L^{2}_{\xi,\tau}}.
\end{align*}
On the other hand, if $p_{a}(\xi) \geq 1$, then we proceed as in \eqref{disslinear1} to find that
\begin{align*}
\sum_{k \geq 0} 2^{k/2} \| \chi_{A_{j}} P_{k}(h_{2}) \|_{L^{2}_{\xi,t}}
&\lesssim \sum_{k \geq 0} 2^{k/2} \sup_{\xi \in A_{j}} \| \chi_{A_{j}} P_{k} \big ( \eta(t)(1 - e^{-p_{a}(\xi)t}) \big )\|_{L^{2}_{t}} \int_{\vert \tau \vert \leq 1} \frac{\| \chi_{A_{j}} \widetilde{w}(\tau,\xi) \|_{L^{2}_{\xi}}}{\langle \tau \rangle} d\tau\\
&\lesssim \sum_{k \geq 0} 2^{-k/2} \| \chi_{A_{j}} \chi_{B_{k}} \widetilde{w}(\tau,\xi) \|_{L^{2}_{\xi,\tau}}.
\end{align*}

%{\color{red}{This last part needs to be explained. It is similar to the analysis for $h_{4}$ below. This estimate needs to be moved up and then referenced in the $h_{4}$ estimate.}}

\noindent{\textbf{Estimate for $h_{3}$.}} Let
\begin{equation*}
g(\tau,\xi) = \frac{\chi_{\vert \tau \vert \geq 1} \widetilde{w}(\tau,\xi)}{i\tau + p_{a}(\xi)}.
\end{equation*}
Observe that
\begin{align*}
2^{k/2} \| \chi_{A_{j}} P_{k}(h_{3}) \|_{L^{2}_{\xi,t}} &\lesssim 2^{k/2} \| \chi_{A_{j}} \chi_{B_{k}} \left ( \widehat{\rho}(\tau) \ast_{\tau} g(\tau,\xi) \right ) \|_{L^{2}_{\tau,\xi}}\\
&\lesssim 2^{k/2} \left \| \| \chi_{A_{j}} \widetilde{w}(\tau,\xi) \|_{L^{2}_{\xi}} \frac{\chi_{B_{k}}}{\langle \tau \rangle} \chi_{\vert \tau \vert \geq 1} \right \|_{L^{2}_{\tau}}\\
&\lesssim 2^{-k/2} \| \chi_{A_{j}} \chi_{B_{k}} \widetilde{w} \|_{L^{2}_{\xi,\tau}},
\end{align*}
from which the desired estimate follows.

\noindent{\textbf{Estimate for $h_{4}$.}} Consider
\begin{align*}
\| \chi_{A_{j}} P_{k}(h_{4}) \|_{L^{2}_{\xi,t}}
&= \left \| P_{k} \left ( \rho(t) \int_{\vert \tau \vert \geq 1} \frac{e^{-p_{a}(\xi)t}}{i\tau + p_{a}(\xi)} \chi_{A_{j}} \widetilde{w}(\tau,\xi) d\tau \right ) \right \|_{L^{2}_{\xi,t}}\\
&\leq \sup_{\xi \in A_{j}} \| P_{k} ( \rho(t) e^{-p_{a}(\xi)t} ) \|_{L^{2}_{t}} \left \| \int_{\vert \tau \vert \geq 1} \frac{\chi_{A_{j}} \widetilde{w}(\tau,\xi)}{i\tau + p_{a}(\xi)} d\tau \right \|_{L^{2}_{\xi}}\\
&\leq \sup_{\xi \in A_{j}} \| P_{k} ( \rho(t) e^{-p_{a}(\xi)t} ) \|_{L^{2}_{t}} \int_{\vert \tau \vert \geq 1} \frac{\| \chi_{A_{j}} \widetilde{w}(\tau,\xi) \|_{L^{2}_{\xi}}}{\vert \tau \vert} d\tau
\end{align*}
Using a Taylor expansion in the case when $j = 0$ as in \eqref{homeq1}, and \eqref{disslinear1} in the case when $j \geq 1$, we have
\begin{equation*}
\sum_{k \geq 0} 2^{k/2} \| \chi_{A_{j}} P_{k}(h_{4}) \|_{L^{2}_{\xi,t}} \lesssim \int_{\vert \tau \vert \geq 1} \frac{\| \chi_{A_{j}} \widetilde{w}(\tau,\xi) \|_{L^{2}_{\xi}}}{\vert \tau \vert} d\tau.
\end{equation*}
Note that
\begin{align*}
\int_{\vert \tau \vert \geq 1} \frac{\| \chi_{A_{j}} \widetilde{w}(\tau,\xi) \|_{L^{2}_{\xi}}}{\vert \tau \vert} d\tau \leq \sum_{k \geq 0} \left ( \int_{\vert \tau \vert \geq 1} \frac{\chi_{B_{k}}}{\vert \tau \vert^{2}} d\tau \right )^{1/2} \| \chi_{A_{j}} \chi_{B_{k}} \widetilde{w} \|_{L^{2}_{\xi,\tau}} \lesssim \sum_{k \geq 0} 2^{-k/2} \| \chi_{A_{j}} \chi_{B_{k}} \widetilde{w} \|_{L^{2}_{\xi,\tau}},
\end{align*}
from which the desired estimate follows.
\end{proof}

\subsection{Bilinear estimate}
To prove \eqref{targetcontrol} we will require the bilinear estimate
\begin{equation*}
\| u_{x} v \|_{X^{1, -1/2, 1}} \lesssim \| u \|_{X^{1,1/2,1}} \| v \|_{X^{1,1/2,1}}.
\end{equation*}
We will prove a more general result:
\begin{equation}
\label{desiredbilinear}
\| u_{x} v \|_{X^{s, -1/2, 1}} \lesssim \| u \|_{X^{s,1/2,1}} \| v \|_{X^{s,1/2,1}},
\end{equation}
provided $s \geq 0$.
Rewriting this in frequency variables we see that \eqref{desiredbilinear} is equivalent to
\begin{equation}
\label{Ftargetcontrol}
\| (\vert \xi_{1} \vert f) \ast g \|_{\widetilde{X}^{s,-1/2,1}} \lesssim \| f \|_{\widetilde{X}^{s, 1/2,1}} \| g \|_{\widetilde{X}^{s, 1/2,1}},
\end{equation}
where the space $\widetilde{X}^{s,1/2,1}$ is as in \eqref{Xs,1/2,1}.
In this direction we first have the following basic bilinear estimates, the proofs of which are given in \cite{MR2501679}.

\begin{lemma}
Suppose that $\supp f, \supp g \subseteq A_{j}$. Then
\begin{equation}
\label{basicbilinear1}
\left \| \vert \xi \vert^{1/4} f \ast g \right \|_{L^{2}_{\tau,\xi}} \lesssim \| f \|_{\widetilde{X}^{0,1/2,1}} \| g \|_{\widetilde{X}^{0,1/2,1}}.
\end{equation}
If
\begin{equation*}
K := \inf \{ \vert \xi_{1} - \xi_{2} \vert \ \vert \ \exists \tau_{1}, \tau_{2} \ \text{such that}\ (\tau_{1}, \xi_{1}) \in \supp f, (\tau_{2}, \xi_{2}) \in \supp g \} > 0,
\end{equation*}
then
\begin{equation}
\label{basicbilinear2}
\left \| \vert \xi \vert^{1/2} f \ast g \right \|_{L^{2}_{\tau,\xi}} \lesssim K^{-1/2} \| f \|_{\widetilde{X}^{0,1/2,1}} \| g \|_{\widetilde{X}^{0,1/2,1}}.
\end{equation}
\end{lemma}

\begin{lemma}
Suppose that $\supp f \subseteq A_{j}$ and let $g$ be an arbitrary test function. For $k \geq 0$ we have
\begin{equation}
\label{basicbilinear3}
\| f \ast g \|_{L^{2}(B_{k})} \lesssim 2^{k/4} \| f \|_{\widetilde{X}^{0,1/2,1}} \| \vert \xi \vert^{-1/4} g \|_{L^{2}_{\tau,\xi}}.
\end{equation}
If $\Omega \subseteq \R^{2}$ satisfies
\begin{equation*}
K := \inf \{ \vert \xi + \xi_{1} \vert \ \vert \ \exists \tau,\tau_{1} \ \text{such that} \ (\tau,\xi) \in \Omega, (\tau_{1},\xi_{1}) \in \supp f \} > 0,
\end{equation*}
then
\begin{equation}
\label{basicbilinear4}
\| f \ast g \|_{L^{2}(\Omega \cap B_{k})} \lesssim 2^{k/2} K^{-1/2} \| f \|_{\widetilde{X}^{0,1/2,1}} \| \vert \xi \vert^{-1/2} g \|_{L^{2}_{\tau,\xi}}.
\end{equation}
\end{lemma}

We are now prepared to establish our bilinear estimate.

\begin{proposition}
\label{bilinearestimate}
Suppose that $f,g \in \widetilde{X}^{s,1/2,1}$ with $s \geq 0$. Then
\begin{equation}
\label{bilinear}
\| (\vert \xi_{1} \vert f) \ast g \|_{\widetilde{X}^{s,-1/2,1}} \lesssim \| f \|_{\widetilde{X}^{s,1/2,1}} \| g \|_{\widetilde{X}^{s,1/2,1}}.
\end{equation}
\end{proposition}

\begin{proof}
We may divide $f$ and $g$ into components as follows: Define $f_{j_{1}, k_{1}} := \chi_{A_{j_{1}}} \chi_{B_{k_{1}}} f$ and $g_{j_{2}, k_{2}} := \chi_{A_{j_{2}}} \chi_{B_{k_{2}}} g$. We thus have
\begin{equation*}
f = \sum_{j_{1} \geq 0} \sum_{k_{1} \geq 0} f_{j_{1},k_{1}} \qquad \text{and} \qquad g = \sum_{j_{2} \geq 0} \sum_{k_{2} \geq 0} g_{j_{2}, k_{2}}.
\end{equation*}
Our goal is to estimate
\begin{equation}
\label{bilinear1}
\sum_{j \geq 0} 2^{2sj} \left ( \sum_{k \geq 0} \sum_{j_{1} \geq 0} \sum_{k_{1} \geq 0} \sum_{j_{2} \geq 0} \sum_{k_{2} \geq 0} 2^{-k/2} 2^{j_{1}} \| f_{j_{1},k_{1}} \ast g_{j_{2}, k_{2}} \|_{L^{2}(A_{j} \cap B_{k})} \right )^{2};
\end{equation}
indeed we aim to establish an estimate of the form
\begin{equation*}
\text{\eqref{bilinear1}} \lesssim \| f \|_{\widetilde{X}^{s,1/2,1}}^{2} \| g \|_{\widetilde{X}^{s,1/2,1}}^{2}.
\end{equation*}
It suffices to prove \eqref{bilinear1} in the following cases:
\begin{enumerate}
\item[(1)] At least two of $j, j_{1}, j_{2}$ are less than 20.
\item[(2)] $j_{1}, j_{2} \geq 20$ and $j < j_{1} - 10$.
\item[(3)] $j, j_{1} \geq 20$, $\vert j - j_{1} \vert \leq 10$.
\end{enumerate}

To simplify our notation below we let
\begin{equation*}
F_{j_{1},k_{1}} := 2^{j_{1}s} 2^{k_{1}/2} \| f_{j_{1},k_{1}} \|_{L^{2}}, \qquad \text{and} \qquad G_{j_{2},k_{2}} := 2^{j_{2}s} 2^{k_{2}/2} \| g_{j_{2},k_{2}} \|_{L^{2}(\R^{2})}.
\end{equation*}
%{\color{red}{Write the desired estimate in terms of $F_{j_{1},k_{1}}$ and $G_{j_{2},k_{2}}$ and shorten the estimates below accordingly.}}

\smallskip

\noindent{\textbf{Case (1).}} We may assume that $j, j_{1}, j_{2} \leq 30$. We apply Young's inequality followed by H\"{o}lder's inequality to see that
\begin{align*}
\| f_{j_{1}, k_{1}} \ast g_{j_{2}, k_{2}} \|_{L^{2}_{\tau,\xi}} &\leq \| f_{j_{1}, k_{1}} \|_{L^{2}_{\xi_{1}} L^{4/3}_{\tau_{1}}} \| g_{j_{2}, k_{2}} \|_{L^{1}_{\xi_{2}} L^{4/3}_{\tau_{2}}}\\
&\lesssim \left ( 2^{15k_{1}/32} \| f_{j_{1}, k_{1}} \|_{L^{2}_{\xi_{1},\tau_{1}}} \right ) \left ( 2^{j_{2}/2} 2^{15k_{2}/32} \| g_{j_{2}, k_{2}} \|_{L^{2}_{\xi_{2},\tau_{2}}} \right ).
\end{align*}
After summing in $k$ and summing over $j$ (a finite sum), we are thus left with
\begin{align*}
\text{\eqref{bilinear1}} &\lesssim \left ( \sum_{j_{1} = 0}^{30} \sum_{k_{1} \geq 0} 2^{j_{1}} 2^{15k_{1}/32} \| f_{j_{1}, k_{1}} \|_{L^{2}_{\xi_{1}, \tau_{1}}} \right )^{2}  \left ( \sum_{j_{2} = 0}^{30} \sum_{k_{2} \geq 0} 2^{j_{2}/2} 2^{15k_{2}/32} \| g_{j_{2}, k_{2}} \|_{L^{2}_{\xi_{2}, \tau_{2}}} \right )^{2}
\end{align*}
Observe that since the sum in $j_{2}$ is finite, we have
\begin{align*}
&\sum_{j_{2}=0}^{30} \sum_{k_{2} \geq 0} 2^{j_{2}/2} 2^{15k_{2}/32} \| g_{j_{2},k_{2}} \|_{L^{2}_{\tau_{2},\xi_{2}}}\\
= &\sum_{j_{2} = 0}^{30} \sum_{k_{2} \geq 0} 2^{sj_{2}} 2^{(1/2 - s)j_{2}} 2^{15k_{2}/32} \| g_{j_{2},k_{2}} \|_{L^{2}_{\tau_{2},\xi_{2}}}\\
& \leq \left ( \sum_{j_{2} = 0}^{30} 2^{2(1/2 - s)j_{2}} \right )^{1/2}  \left ( \sum_{j_{2} = 0}^{30} \left ( \sum_{k_{2} \geq 0} 2^{sj_{2}} 2^{15k_{2}/32} \| g_{j_{2}, k_{2}} \|_{L^{2}_{\tau_{2},\xi_{2}}} \right )^{2} \right )^{1/2}\\
&\lesssim \| g \|_{\widetilde{X}^{s,1/2,1}}.
\end{align*}
A similar argument can be used to show that
\begin{equation*}
\sum_{j_{1} = 0}^{30} \sum_{k_{1} \geq 0} 2^{j_{1}} 2^{15k_{1}/32} \| f_{j_{1}, k_{1}} \|_{L^{2}_{\xi_{1}, \tau_{1}}} \lesssim \| f \|_{\widetilde{X}^{s,1/2,1}},
\end{equation*}
thereby yielding the desired estimate.

\smallskip

\noindent{\textbf{Case (2).}} Here we may assume that $\vert j_{1} - j_{2} \vert \leq 1$, for otherwise $f_{j_{1}} \ast g_{j_{2}} = 0$ in $A_{j}$. For $(\tau_{1},\xi_{1}) \in A_{j_{1}} \cap B_{k_{1}}$ and $(\tau_{2},\xi_{2}) \in A_{j_{2}} \cap B_{k_{2}}$ we have
\begin{equation}
\label{algebraid}
(\tau_{1} + \tau_{2}) - (\xi_{1} + \xi_{2})^{3} - (\tau_{1} - \xi_{1}^{3}) - (\tau_{2} - \xi_{2}^{3}) = -3\xi \xi_{1} \xi_{2}.
\end{equation}
It follows that $f_{j_{1},k_{1}} \ast g_{j_{2},k_{2}} = 0$ on $A_{j} \cap B_{k}$ unless
\begin{equation*}
2^{k_{max}} \gtrsim 2^{j} 2^{j_{1}} 2^{j_{2}} \sim 2^{j + 2j_{1}},
\end{equation*}
where $k_{max} = \max \{ k, k_{1}, k_{2} \}$.

Suppose that $k = k_{max}$, meaning that $2^{-k/2} \lesssim 2^{-j/2 - j_{1}}$. Notice that in order for $f_{j_{1}} \ast g_{j_{2}}$ to have low frequency support we require that $\xi_{1}$ and $\xi_{2}$ must have opposite signs for $\xi_{1} \in \supp f_{j_{1}}, \xi_{2} \in \supp g_{j_{2}}$. It follows that $\supp f_{j_{1}}$ and $\supp g_{j_{2}}$ are separated by $K \sim 2^{j_{1}}$. In light of \eqref{basicbilinear2} we have
\begin{align*}
2^{j/2} \| f_{j_{1},k_{1}} \ast g_{j_{2},k_{2}} \|_{L^{2}(A_{j} \cap B_{k})} \lesssim 2^{-j_{1}/2} 2^{-j_{1}s} 2^{-j_{2}s} F_{j_{1},k_{1}} G_{j_{2},k_{2}}.
\end{align*}
Thus
\begin{align*}
\text{\eqref{bilinear1}} &\lesssim \sum_{j \geq 0} \left ( \sum_{j_{1} \geq j + 11} \sum_{k_{1} \geq 0} \sum_{j_{2} = j_{1} - 1}^{j_{1} + 1} \sum_{k_{2} \geq 0} 2^{sj - j} 2^{-j_{1}/2 - j_{1}s} 2^{-j_{2}s} F_{j_{1},k_{1}} G_{j_{2},k_{2}} \right )^{2}\\
&\lesssim \sum_{j \geq 0} 2^{-2sj - 3j} \left ( \sum_{j_{1} \geq 0} \sum_{k_{1} \geq 0} \sum_{j_{2} \geq 0} \sum_{k_{2} \geq 0} F_{j_{1},k_{1}} G_{j_{2},k_{2}} \right )^{2}\\
&\lesssim \| f \|_{X^{s,1/2,1}}^{2} \| g \|_{X^{s,1/2,1}}^{2}.
\end{align*}

Next we suppose that $k_{1} = k_{max}$, so $2^{k_{1}} \gtrsim 2^{j + 2j_{1}}$.  We use \eqref{basicbilinear4} with $K \sim 2^{j_{1}}$ to get that
\begin{align*}
\| f_{j_{1},k_{1}} \ast g_{j_{2},k_{2}} \|_{L^{2}(A_{j} \cap B_{k})} &\lesssim 2^{k/2} 2^{-j_{1}} 2^{k_{2}/2} \| f_{j_{1},k_{1}} \|_{L^{2}} \| g_{j_{2},k_{2}} \|_{L^{2}}\\
&\lesssim 2^{k/2} 2^{-j_{1}} 2^{-j_{1}s} 2^{-j_{2}s} 2^{-k_{1}/2} F_{j_{1},k_{1}} G_{j_{2},k_{2}}.
\end{align*}
Therefore
\begin{align*}
\text{\eqref{bilinear1}} \lesssim \sum_{j \geq 0} \left (\sum_{k = 0}^{k_{1}} \sum_{j_{1} \geq j + 11} \sum_{k_{1} \geq j + 2j_{1}} \sum_{j_{2} = j_{1} - 1}^{j_{1} + 1} \sum_{k_{2} = 0}^{k_{1}} 2^{js -k_{1}/2 - j_{1}s - j_{2}} F_{j_{1},k_{1}} G_{j_{2},k_{2}} \right )^{2}.
\end{align*}
Now we estimate
\begin{equation*}
2^{js - k_{1}/2 - j_{1}s - j_{2}s} \lesssim 2^{-js} 2^{-k/16} 2^{-7(j + 2j_{1})/16} \lesssim 2^{-js-21j/16} 2^{-k/16}.
\end{equation*}
It follows that
\begin{equation*}
\text{\eqref{bilinear1}} \lesssim \sum_{j \geq 0} 2^{-js - 21j/16} \left ( \sum_{k \geq 0} \sum_{j_{1} \geq 0} \sum_{k_{1} \geq 0} \sum_{j_{2} \geq 0} \sum_{k_{2} \geq 0} 2^{-k/16} F_{j_{1},k_{1}} G_{j_{2},k_{2}} \right )^{2},
\end{equation*}
and the desired estimate follows.

Finally we suppose that $k_{2} = k_{max}$. Since $\vert j_{1} - j_{2} \vert \leq 1$ we may proceed in the same way as the case when $k_{1} = k_{max}$ to obtain the desired estimate.

\noindent{\textbf{Case (3).}} We may assume that $j_{2} \leq j + 11$. Returning to \eqref{algebraid} we require $2^{k_{max}} \gtrsim 2^{2j + j_{2}}$. We first suppose that $k = k_{max}$. In this case we use \eqref{basicbilinear1} to see that
\begin{align*}
 \| f_{j_{1},k_{1}} \ast g_{j_{2},k_{2}} \|_{L^{2}(A_{j} \cap B_{k})} &\lesssim 2^{-j/4} 2^{k_{1}/2} 2^{k_{2}/2} \| f_{j_{1},k_{1}} \|_{L^{2}} \| g_{j_{2},k_{2}} \|_{L^{2}}\\
 &\lesssim 2^{-j/4} 2^{-j_{1}s} 2^{-j_{2}s} F_{j_{1},k_{1}} G_{j_{2},k_{2}}.
\end{align*}
Therefore
\begin{align*}
\text{\eqref{bilinear1}} &\lesssim \sum_{j \geq 0} \left ( \sum_{k \geq 2j + j_{2}} \sum_{j_{1} = j-10}^{j+10} \sum_{k_{1} = 0}^{k} \sum_{j_{2} = 0} \sum_{k_{2} = 0}^{k} 2^{js} 2^{j_{1}} 2^{-k/2} 2^{-j/4} 2^{-j_{1}s} 2^{-j_{2}s} F_{j_{1},k_{1}} G_{j_{2},k_{2}} \right )^{2}\\
&\lesssim \sum_{j \geq 0} \left ( \sum_{j_{1} = j-10}^{j+10} \sum_{k_{1} = 0}^{k} \sum_{j_{2} = 0}^{j + 11} \sum_{k_{2} = 0}^{k} 2^{js} 2^{j_{1}} 2^{-j - j_{2}/2} 2^{-j/4} 2^{-j_{1}s} 2^{-j_{2}s} F_{j_{1},k_{1}} G_{j_{2},k_{2}} \right )^{2}\\
&\lesssim \sum_{j \geq 0} 2^{-j/2} \left ( \sum_{j_{1} \geq 0} \sum_{k_{1} \geq 0} \sum_{j_{2} \geq 0} \sum_{k_{2} \geq 0} 2^{-j_{2}s - j_{2}/2} F_{j_{1},k_{1}} G_{j_{2},k_{2}} \right )^{2},
\end{align*}
which is sufficient.

Suppose that $k_{max} = k_{1}$, so that $2^{k_{1}} \gtrsim 2^{2j + j_{2}}$. We use \eqref{basicbilinear3} to estimate
\begin{align*}
\| f_{j_{1},k_{1}} \ast g_{j_{2},k_{2}} \|_{L^{2}(A_{j} \cap B_{k})} &\lesssim 2^{k/4} 2^{-j_{1}/4} 2^{k_{2}/2} \| f_{j_{1},k_{1}} \|_{L^{2}} \| g_{j_{2},k_{2}} \|_{L^{2}}\\
&\lesssim 2^{k/4} 2^{-j_{1}/4} 2^{-j_{1}s} 2^{-j_{2}s} 2^{-k_{1}/2} F_{j_{1},k_{1}} G_{j_{2},k_{2}}.
\end{align*}
It follows that
\begin{align*}
\text{\eqref{bilinear1}} &\lesssim \sum_{j \geq 0} \left ( \sum_{k = 0}^{k_{1}} \sum_{j_{1} = j-10}^{j + 10} \sum_{k_{1} \geq 2j + j_{2}} \sum_{j_{2} = 0}^{j+11} \sum_{k_{2} = 0}^{k_{1}} 2^{js} 2^{-k/4} 2^{3j_{1}/4} 2^{-k_{1}/2} 2^{-j_{1}s} 2^{-j_{2}s} F_{j_{1},k_{1}} G_{j_{2},k_{2}} \right )^{2}\\
&\lesssim \sum_{j \geq 0} 2^{-j/4} \left ( \sum_{j_{1} \geq 0} \sum_{k_{1} \geq 0} \sum_{j_{2} \geq 0} \sum_{k_{2} \geq 0} 2^{-j_{2}/2} 2^{-j_{2}s} F_{j_{1},k_{1}} G_{j_{2},k_{2}} \right )^{2},
\end{align*}
which is sufficient.

Finally, suppose that $k_{max} = k_{2}$. Here we divide our analysis into the following two cases:
\begin{enumerate}
\item[(i)] $\vert j_{2} - j \vert \leq 5$;
\item[(ii)] $\vert j_{2} - j \vert > 5$.
\end{enumerate}
In Case (i) we use \eqref{basicbilinear3} as above to see that
\begin{align*}
\| f_{j_{1},k_{1}} \ast g_{j_{2},k_{2}} \|_{L^{2}(A_{j} \cap B_{k})} \lesssim 2^{k/4} 2^{-j_{2}/4} 2^{-j_{1}s} 2^{-j_{2}s} 2^{-k_{2}/2} F_{j_{1},k_{1}} G_{j_{2},k_{2}}.
\end{align*}
We thereby find that
\begin{align*}
\text{\eqref{bilinear1}} &\lesssim \sum_{j \geq 0} \left ( \sum_{k=0}^{k_{2}} \sum_{j_{1} = j - 10}^{j + 10} \sum_{k_{1} = 0}^{k_{2}} \sum_{\substack{j_{2} \geq 0\\ \vert j - j_{2} \vert \leq 5}} \sum_{k_{2} \geq 2j + j_{2}} 2^{js} 2^{j_{1}} 2^{-k/4} 2^{-j_{2}/4} 2^{-k_{2}/2} 2^{-j_{1}s} 2^{-j_{2}s} F_{j_{1},k_{1}} G_{j_{2},k_{2}} \right )^{2}\\
&\lesssim \sum_{j \geq 0} \left ( \sum_{j_{1} = j-10}^{j+10} \sum_{k_{1} \geq 0} \sum_{\substack{j_{2} \geq 0\\ \vert j - j_{2} \vert \leq 5}} \sum_{k_{2} \geq 0} 2^{-j_{2}/4} 2^{j_{1}} 2^{-j} 2^{-j_{2}/2} F_{j_{1},k_{1}} G_{j_{2},k_{2}} \right )^{2}\\
&\lesssim \sum_{j \geq 0} 2^{-3j/4} \left ( \sum_{j_{1} \geq 0} \sum_{k_{1} \geq 0} \sum_{j_{2} \geq 0} \sum_{k_{2} \geq 0} F_{j_{1},k_{1}} G_{j_{2},k_{2}} \right )^{2},
\end{align*}
which is sufficient for our purposes.
In Case (ii) we can use \eqref{basicbilinear4} with $K \sim 2^{j}$ to estimate
\begin{equation*}
\| f_{j_{1},k_{1}} \ast g_{j_{2},k_{2}} \|_{L^{2}(A_{j} \cap B_{k})} \lesssim 2^{k/2} 2^{-j/2} 2^{-j_{2}/2} 2^{-j_{1}s} 2^{-j_{2}s} 2^{-k_{2}/2} F_{j_{1},k_{1}} G_{j_{2},k_{2}}.
\end{equation*}
We use
\begin{equation*}
2^{-k_{2}/2} \lesssim 2^{-k/16} 2^{-7k_{2}/16} \lesssim 2^{-k/16} 2^{-7j/8} 2^{-7j_{2}/17}
\end{equation*}
to see that
\begin{align*}
\text{\eqref{bilinear1}} &\lesssim \sum_{j \geq 0} \left ( \sum_{k=0}^{k_{2}} \sum_{j_{1} = j-10}^{j + 10} \sum_{k_{1} = 0}^{k_{2}} \sum_{\substack{j_{2} \geq 0\\ \vert j_{2} < j - 5}} \sum_{k_{2} \geq 0} 2^{js} 2^{j_{1}} 2^{-11j/8} 2^{j_{1}} 2^{-17j_{2}/16} 2^{-j_{1}s} 2^{-j_{2}s} 2^{-k/16} F_{j_{1},k_{1}} G_{j_{2},k_{2}} \right )^{2}\\
&\lesssim \sum_{j \geq 0} 2^{-3j/4} \left ( \sum_{j_{1} \geq 0} \sum_{k_{1} \geq 0} \sum_{j_{2} \geq 0} \sum_{k_{2} \geq 0} F_{j_{1},k_{1}} G_{j_{2},k_{2}} \right )^{2}.
\end{align*}
This completes the proof of the Proposition.
\end{proof}

\subsection{Local-in-time control of the perturbations}

 %Next we establish the desired control on the perturbations, as in \eqref{targetcontrol}.

%\input{vwlocal}
The purpose of this subsection is to establish estimates of the form \eqref{targetcontrol}. Before stating a proposition to this effect, we note that from the modulation equations \eqref{modulation} we have
\begin{equation}
\label{parametercontrol1}
\vert \dot{c} \vert, \vert \dot{\gamma} \vert \lesssim \| w \|_{L^{\infty}_{t} H^{1}_{x}} \lesssim  \| w \|_{X^{1,1/2,1}}.
\end{equation}
We also require control over $\vert c(t) - c_{0} \vert$, which is obtained by integrating the control on $\dot{c}(t)$:
\begin{align*}
\vert c(t) - c_{0} \vert \leq \int_{0}^{t} \vert \dot{c}(\tau) \vert d\tau \lesssim \int_{0}^{t} \| w(\tau) \|_{H^{1}_{x}} d\tau \lesssim \| w \|_{L^{1}_{t} H^{1}_{x}}.
\end{align*}
Since we have restricted $t \in [0,\delta]$, H{\"o}lder's inequality gives
\begin{equation}
\label{parametercontrol2}
\vert c(t) - c_{0} \vert \leq \delta^{1/2} \| w \|_{L^{2}_{t} H^{1}_{x}} \lesssim \delta^{1/2} \| w \|_{X^{1,1/2,1}_{\delta}}.
\end{equation}

\begin{proposition}
\label{vwlwp}
Let $0 < a < \sqrt{c_{0}/3}$.
There is an $r > 0$ such that the following statement holds: If $v_{0} \in H^{1}(\R)$ satisfies $\| v_{0} \|_{H^{1}} < r$ and $\| w_{0} \|_{H^{1}} < r$ where $w_{0}(x) = e^{ax} v_{0}(x)$, then there is a $\delta > 0$ so that the equations \eqref{vequation} and \eqref{wequationmod} admit solutions $v(t,x),w(t,x)$, respectively, on $[0,\delta]$. Moreover, these solutions satisfy
\begin{equation}
\label{vwlwpestimates}
\| v \|_{X^{1,1/2,1}_{\delta}} \lesssim \| v_{0} \|_{H^{1}}, \qquad \text{and} \qquad \| w \|_{X^{1,1/2,1}_{\delta}} \lesssim \| w_{0} \|_{H^{1}}.
\end{equation}
\end{proposition}

\begin{proof}
We begin with the equation for $v$, given by \eqref{vequation}. Changing variables $x \mapsto x - (\gamma - 2c_{0}t + \int_{0}^{t} c(s) ds)$ leaves us with
\begin{equation}
\label{vequation1}
\partial_{t} + \partial_{x}^{3}v + (\dot{\gamma} \partial_{x} + \dot{c} \partial_{c}) \psi_{c(t)} + 2 \partial_{x} (\psi_{c_{0}} v ) - \partial_{x}(v^{2}) = 0.
\end{equation}
This can be rewritten as an integral equation using Duhamel's formula:
\begin{equation*}
%\label{vduhamel}
v(t,x) = W_{1}(t) v_{0}(x) + \int_{0}^{t} W_{1}(t-s) \big ( (\dot{\gamma}\partial_{x} + \dot{c}\partial_{c}) \psi_{c(t)} + 2 \partial_{x}(\psi_{c_{0}}v) - \partial_{x}(v^{2}) \big ) ds.
\end{equation*}
We will show that the operator $\Phi$ given by
\begin{equation}
\label{vduhamel}
\Phi v = \rho(t) W_{1}(t) v_{0} + \rho(t) \int_{0}^{t} W_{1}(t-s) \big ( (\dot{\gamma} \partial_{x} + \dot{c} \partial_{c}) \rho(s) \psi_{c(s)} + 2 \partial_{x}(\rho^{2}(s)\psi_{c_{0}}v) - \partial_{x}(\rho^{2}(s) v^{2}) \big ) ds,
\end{equation}
is a contraction on a ball that is to be chosen momentarily, and where $\eta$ is a smooth cutoff adapted to the time interval $[0,\delta]$. We now estimate
\begin{align*}
\| \Phi v \|_{X^{1,1/2,1}_{\delta}} &\lesssim \| v_{0} \|_{H^{1}} + \left \|  \int_{0}^{t} W_{1}(t-s) \big ( (\dot{\gamma} \partial_{x} + \dot{c} \partial_{c}) \rho \psi_{c}ds \big ) \right \|_{X^{1,1/2,1}_{\delta}}\\
&\quad + \left \| \int_{0}^{t} W_{1}(t-s) \partial_{x}(\rho^{2}\psi_{c_{0}}v)ds \right \|_{X^{1,1/2,1}_{\delta}} + \left \|\int_{0}^{t} W_{1}(t-s) \partial_{x}(\rho^{2} v^{2}) \big ) ds   \right \|_{X^{1,1/2,1}_{\delta}}\\
&=: \| v_{0} \|_{H^{1}_{x}} + (I) + (II) + (III).
\end{align*}
Before proceeding further, observe that as a consequence of the embedding $X^{1,\frac{1}{2} + \epsilon} \hookrightarrow X^{1,1/2,1}$ for any $\epsilon > 0$ and the standard inequality (see \cite{MR2492151}, for instance)
\begin{equation*}
\| u \|_{X^{s,\frac{1}{2} + \epsilon}_{\delta}} \lesssim \delta^{\epsilon} \| u \|_{X^{s,\frac{1}{2} + 2\epsilon}_{\delta}},
\end{equation*}
we have that
\begin{equation*}
\| \psi \|_{X^{1,1/2,1}_{\delta}} \lesssim \delta^{\epsilon},
\end{equation*}
provided $\delta > 0$ and $\epsilon > 0$  chosen sufficiently small.
We estimate $(I),(II),$ and $(III)$ using Proposition \ref{bilinearestimate} along with Lemmas \ref{linearairy} and \ref{lineardissipative}:
\begin{align*}
(I) & \lesssim \left ( \vert \dot{\gamma} \vert + \vert \dot{c} \vert \right ) \| \psi \|_{X^{1,1/2,1}_{\delta}} \lesssim \delta^{\epsilon} \| w \|_{X^{1,1/2,1}_{\delta}};\\
(II) &\lesssim \| \partial_{x} (\psi_{c_{0}} v) \|_{X^{1,-1/2,1}_{\delta}} \lesssim \delta^{\epsilon} \| v \|_{X^{1,1/2,1}_{\delta}};\\
(III) &\lesssim \| \partial_{x} (v^{2}) \|_{X^{1,-1/2,1}_{\delta}} \lesssim \| v \|_{X^{1,1/2,1}_{\delta}}^{2}.
\end{align*}
It follows that
\begin{equation}
\| \Phi v \|_{X^{1,1/2,1}_{\delta}} \lesssim \| v_{0} \|_{H^{1}} + \| w \|_{X^{1,1/2,1}_{\delta}} + \| v \|_{X^{1,1/2,1}_{\delta}} + \| v \|_{X^{1,1/2,1}_{\delta}}^{2}.
\end{equation}

%{\color{red}{Is the correct term $e^{ax} (\dot{\gamma} \partial_{x} + \dot{c} \partial_{c}) e^{-ax} Q_{c}$? Or should it be absent the $e^{-ax}$ factor?}}

Turning to the $w$-equation \eqref{wequationmod}, we again change variables with an eye toward removing the first-order term: let $x \mapsto x - (c_{0} - 3a^{2} + a)t + \int_{0}^{t} \dot{c}(s) ds - \gamma(t)$. The equation then reads
\begin{align*}
&\partial_{t} w + \partial_{x}^{3} w - 3a \partial_{x}^{2} w + a (c_{0} - a^{2}) w + a \dot{\gamma}w - e^{ax} (\dot{c} \partial_{c} + \dot{\gamma} \partial_{x}) \psi_{c} - e^{ax} \partial_{x}(v^{2}) - a(c-c_{0})w\\
&\qquad + \langle \mathcal{F}(t),\eta_{1} \rangle \zeta_{1} + \langle \mathcal{F}(t),\eta_{2} \rangle \zeta_{2} = 0,
\end{align*}
where
\begin{equation*}
\mathcal{F}(t) = e^{ax} ( \dot{\gamma} \partial_{x} + \dot{c} \partial_{c}) e^{-ax} \psi_{c} - a \dot{\gamma} w + e^{ax} \partial_{x} (v^{2}) - a(c-c_{0})w.
\end{equation*}
We will show that, along with $\Phi$ defined above, the map $\Psi$ defined by
\begin{align*}
\Psi w = \rho(t) W_{2}(t) w_{0} + \rho(t) &\int_{0}^{t} W_{2}(t-s) \big ( 2(\partial_{x} - a) (\rho^{2} \psi_{c}w) + a \rho \dot{\gamma}w - e^{ax} (\dot{c} \partial_{c} + \dot{\gamma}\partial_{x})\rho \psi_{c} \\
& \qquad - e^{ax} \partial_{x} (\rho^{2} v^{2}) - a(c-c_{0})\rho w + \rho \langle \mathcal{F},\eta_{1} \rangle \xi_{1} + \rho \langle \mathcal{F},\eta_{2} \rangle \xi_{2} \big ) ds.
\end{align*}
is a contraction on an appropriately chosen ball in $X^{1,1/2,1}_{\delta}$. We begin by estimating
\begin{align*}
\| \Psi w \|_{X^{1,1/2,1}_{\delta}} &\lesssim  \| w_{0} \|_{H^{1}} + \| (\partial_{x} - a) (\rho^{2} \psi_{c}w) \|_{X^{1,-1/2,1}_{\delta}} + \| \rho \dot{\gamma} w \|_{X^{1,-1/2,1}_{\delta}}\\
&\quad + \| e^{ax} (\dot{c} \partial_{c} + \dot{\gamma} \partial_{x}) \rho \psi_{c} \|_{X^{1,-1/2,1}_{\delta}} + \| e^{ax} \partial_{x} (\rho^{2} v^{2}) \|_{X^{1,-1/2,1}_{\delta}} + \| (c - c_{0}) \rho w\|_{X^{1,-1/2,1}_{\delta}}\\
&\quad + \| \rho \langle \mathcal{F},\eta_{1} \rangle \zeta_{1} \|_{X^{1,-1/2,1}_{\delta}} + \| \rho \langle \mathcal{F},\eta_{2} \rangle \zeta_{2} \|_{X^{1,-1/2,1}_{\delta}}\\
&=: \| w_{0} \|_{H^{s}} + (I) + (II) + (III) + (IV) + (V) + (VI) + (VII).
\end{align*}

To estimate $(I)$ we use the fact that $e^{ax} \partial_{x} e^{-ax} = \partial_{x} - a$ to see that
\begin{align*}
(I) &= \| e^{ax} \partial_{x} e^{-ax} \rho \psi_{c}w \|_{X^{1,-1/2,1}_{\delta}} \lesssim \| (\partial_{x} \psi_{c}) w \|_{X^{1,-1/2,1}_{\delta}} + \| (e^{ax}\psi_{c}) v_{x} \|_{X^{1,-1/2,1}_{\delta}} \\
&\lesssim \delta^{\epsilon} \| w \|_{X^{1,1/2,1}_{\delta}} + \delta^{\epsilon} \| v \|_{X^{1,1/2,1}_{\delta}}.
\end{align*}
Term $(II)$ is estimated easily using \eqref{parametercontrol1}:
\begin{equation*}
(II) \lesssim \| \dot{\gamma} \|_{L^{\infty}_{t}} \| w \|_{X^{1,-1/2,1}_{\delta}} \lesssim \| w \|_{X^{1,1/2,1}_{\delta}}^{2}.
\end{equation*}
Similarly, we see that
\begin{equation*}
(III) \lesssim (\| \dot{c} \|_{L^{\infty}_{t}} + \| \dot{\gamma} \|_{L^{\infty}_{t}}) \| \psi \|_{X^{1,1/2,1}_{\delta}} \lesssim \delta^{\epsilon} \| w \|_{X^{1,1/2,1}_{\delta}}.
\end{equation*}
Recalling that $w = e^{ax}v$, we have
\begin{equation*}
(IV) = 2 \| w v_{x} \|_{X^{1,-1/2,1}_{\delta}} \lesssim \| w \|_{X^{1,1/2,1}_{\delta}} \| v \|_{X^{1,1/2,1}_{\delta}}.
\end{equation*}
To estimate $(V)$ we use \eqref{parametercontrol2} to see that
\begin{equation*}
(V) \leq \| c - c_{0} \|_{L^{\infty}_{t}} \| w \|_{X^{1,-1/2,1}_{\delta}} \lesssim \| w \|_{X^{1,1/2,1}_{\delta}}^{2}.
\end{equation*}
To estimate $(VI)$ and $(VII)$ we require the following lemma.

\begin{lemma}
Let $f$ be a space-time function and let $s \geq 0$. Then
\begin{equation*}
\| \langle f,\eta_{i} \rangle \zeta_{i} \|_{X^{s,-1/2,1}_{\delta}} \lesssim \| f \|_{X^{s,-1/2,1}_{\delta}}, \qquad i = 1,2.
\end{equation*}
\end{lemma}

\begin{proof}
Let $\widetilde{f}_{j} = \chi_{A_{j}} \widetilde{f}$, as before, so that $f = \sum_{j \geq 0} f_{j}$. Then
\begin{align*}
\| \langle f, \eta_{i} \rangle \xi_{i} \|_{X^{s, -1/2,1}_{\delta}} &\leq \sum_{j \geq 0} \| \langle f_{j}, \eta_{i} \rangle \xi_{i} \|_{X^{s, -1/2,1}_{\delta}}\\
&= \sum_{j \geq 0} \left ( \sum_{n \geq 0} 2^{2ns} \left ( \sum_{k \geq 0} 2^{-k/2} \| \langle \widetilde{f}_{j}, \widehat{\eta}_{i} \rangle \widehat{\xi}_{i} \|_{L^{2}(A_{n} \cap B_{k})} \right )^{2} \right )^{1/2}.
\end{align*}
Note that $\langle \widetilde{f}_{j}, \widehat{\eta}_{i} \rangle = \langle \widetilde{f}_{j}, \chi_{A_{j}} \widehat{\eta}_{i} \rangle$, which is a function of $\tau$ only. Here we denote the Fourier transform of $\eta_{i}, \zeta_{i}$ by $\widehat{\eta}_{i}, \widehat{\zeta}_{i}$, respectively, to emphasize that these are functions of the frequency variable $\xi$ only. It follows that
\begin{align*}
\| \langle \widetilde{f}_{j}, \widehat{\eta}_{i} \rangle \widehat{\zeta}_{i} \|_{L^{2}(A_{n} \cap B_{k})}
&= \| \langle \widetilde{f}_{j}, \chi_{A_{j}} \widehat{\eta}_{i} \rangle \|_{L^{2}_{\tau}(B_{k})} \| \widehat{\zeta}_{i} \|_{L^{2}_{\xi}(A_{n})}\\
&\leq \| \chi_{B_{k}} \widetilde{f}_{j} \|_{L^{2}_{\tau,\xi}} \| \widehat{\eta}_{i} \|_{L^{2}(A_{j})} \| \widehat{\zeta}_{i} \|_{L^{2}_{\xi}(A_{n})}.
\end{align*}
It remains to estimate
\begin{equation}
\label{innerprod1}
\sum_{j \geq 0} \left ( \sum_{n \geq 0} 2^{2ns} \left ( \sum_{k \geq 0} 2^{-k/2} \| \widetilde{f} \|_{L^{2}(A_{j} \cap B_{k})} \| \widehat{\eta}_{i} \|_{L^{2}(A_{j})} \| \widehat{\zeta}_{i} \|_{L^{2}(A_{n})} \right )^{2} \right )^{1/2}.
\end{equation}
In the case when $n \leq j$ we have $2^{2ns} \leq 2^{2js}$ so that
\begin{align*}
\text{\eqref{innerprod1}} &\leq \sum_{j \geq 0} \| \widehat{\eta}_{i} \|_{L^{2}(A_{j})} \left ( \sum_{n = 0}^{j} 2^{2js} \| \widehat{\zeta}_{i} \|_{L^{2}(A_{n})}^{2} \left ( \sum_{k \geq 0} 2^{-k/2} \| \widetilde{f} \|_{L^{2}(A_{j} \cap B_{k})} \right )^{2} \right )^{1/2},\\
\intertext{so that after carrying out the sum in $n$ we have}
&\lesssim \sum_{j \geq 0} \| \widehat{\eta}_{i} \|_{L^{2}(A_{j})} \left ( \sum_{k \geq 0} 2^{js} 2^{-k/2} \| \widetilde{f} \|_{L^{2}(A_{j} \cap B_{k})} \right )
\lesssim \left ( \sum_{j \geq 0} \left ( \sum_{k \geq 0} 2^{js} 2^{-k/2} \| \widetilde{f} \|_{L^{2}(A_{j} \cap B_{k})} \right )^{2} \right )^{1/2},
\end{align*}
where in the last line we've used Cauchy-Schwartz and the fact that $\eta_{i}$ is smooth.
If $n \geq j$, then $2^{2ns} = 2^{2ns} 2^{-2js} 2^{2js}$ and we find that
\begin{align*}
\text{\eqref{innerprod1}} &= \sum_{j \geq 0} 2^{-js} \| \widehat{\eta}_{i} \|_{L^{2}(A_{j})} \left ( \sum_{n \geq j} 2^{2ns} \| \widehat{\zeta}_{i} \|_{L^{2}(A_{n})} \left ( \sum_{k \geq 0} 2^{-k/2} \| \widetilde{f} \|_{L^{2}(A_{j} \cap B_{k})} \right )^{2} \right )^{1/2}\\
\intertext{and after summing in $n$ (using that $\xi_{i}$ is smooth),}
&\lesssim \sum_{j \geq 0} (2^{-js} \| \widehat{\eta}_{i} \|_{L^{2}(A_{j})} \left ( \sum_{k \geq 0} 2^{js} 2^{-k/2} \| \widetilde{f} \|_{L^{2}(A_{j} \cap B_{k})} \right ) \lesssim \left ( \sum_{j \geq 0} \left ( \sum_{k \geq 0} 2^{js} 2^{-k/2} \| \widetilde{f} \|_{L^{2}(A_{j} \cap B_{k})} \right )^{2} \right )^{1/2}.
\end{align*}
\end{proof}

Returning to our estimates, we now have that
\begin{equation*}
(VI), (VII) \lesssim \| \mathcal{F} \|_{X^{1,-1/2,1}_{\delta}} \lesssim \| w \|_{X^{1,1/2,1}_{\delta}}^{2} + \delta^{\epsilon} \| w \|_{X^{1,1/2,1}_{\delta}} + \| w \|_{X^{1,1/2,1}_{\delta}} \| v \|_{X^{1,1/2,1}_{\delta}},
\end{equation*}
following the estimates of $(II)$ through $(V)$.
Taken together, these estimates now give
\begin{equation}
\label{Psiestimate}
\| \Psi w \|_{X^{1,1/2,1}_{\delta}} \lesssim \| w_{0} \|_{H^{1}} + \delta^{\epsilon} \| w \|_{X^{1,1/2,1}_{\delta}} + \delta^{\epsilon} \| v \|_{X^{1,1/2,1}_{\delta}} + \| w \|_{X^{1,1/2,1}_{\delta}}^{2} + \| w \|_{X^{1,1/2,1}_{\delta}} \| v \|_{X^{1,1/2,1}_{\delta}}.
\end{equation}

Suppose that $\| v_{0} \|_{H^{s}}, \| w_{0} \|_{H^{s}} < r \ll 1$, and consider
\begin{equation*}
\mathcal{B} = \left \{ v,w \in X^{1,1/2,1}_{\delta} \ \big \vert \  \| v \|_{X^{1,1/2,1}_{\delta}} \leq 2cr, \| w \|_{X^{1,1/2,1}_{\delta}} \leq 2cr \right \}.
\end{equation*}
According to our estimates we have
\begin{align*}
\| \Phi v \|_{X^{s,1/2,1}_{\delta}} &\leq cr + 4c \delta^{\epsilon} r + 4 c^{2} r^{2},\\
\| \Psi w \|_{X^{s,1/2,1}_{\delta}} &\leq cr + 4c \delta^{\epsilon} r + 4c^{2} r^{2}.
\end{align*}
It follows that if $\delta$ and $r$ are chosen sufficiently small, then the maps $\Phi,\Psi:\mathcal{B} \to \mathcal{B}$.

%{\color{red}{
%\begin{remark}
%As yet we can only make sense of these equations in the case when $s = 1$, since that is where our modulation equations have been defined.
%\end{remark}}}

To see that $\Phi,\Psi$ are contractions on $\mathcal{B}$ we let $v_{1},v_{2} \in \mathcal{B}$ with $w_{1}= e^{ax} v_{1} \in \mathcal{B}, w_{2} = e^{ax} v_{2} \in \mathcal{B}$. Associated with these functions are modulation parameters $(\gamma_{1}, c_{1})$ and $(\gamma_{2},c_{2})$ corresponding to $v_{1},v_{2}$, respectively. From the modulation equations we find that
\begin{equation*}
\| \dot{c}_{1} - \dot{c}_{2} \|_{L^{\infty}_{t}} + \| \dot{\gamma}_{1} + \dot{\gamma}_{2} \|_{L^{\infty}_{t}} \leq \| v_{1} - v_{2} \|_{L^{\infty}_{t} H^{1}_{x}} \| w_{1} + w_{2} \|_{L^{\infty}_{t} H^{1}_{x}} + \| v_{1} + v_{2} \|_{L^{\infty}_{t} H^{1}_{x}} \| w_{1} - w_{2} \|_{L^{\infty}_{t} H^{1}_{x}},
\end{equation*}
where we use the notation $L^{\infty}_{t}$ as shorthand for the space $L^{\infty}_{t \in [0,\delta]}$. Thus we have
\begin{align*}
\| \dot{c}_{1} - \dot{c}_{2} \|_{L^{\infty}_{t}} + \| \dot{\gamma}_{1} + \dot{\gamma}_{2} \|_{L^{\infty}_{t}} &\leq c \| v_{1} - v_{2} \|_{X^{1,1/2,1}_{\delta}} \| w_{1} + w_{2} \|_{X^{1,1/2,1}_{\delta}}\\
&\quad + c \| v_{1} + v_{2} \|_{X^{1,1/2,1}_{\delta}} \| w_{1} - w_{2} \|_{X^{1,1/2,1}_{\delta}}\\
&\leq 4c^{2}r \| v_{1} - v_{2} \|_{X^{1,1/2,1}_{\delta}} + 4c^{2}r \| w_{1} - w_{2} \|_{X^{1,1/2,1}_{\delta}}.
\end{align*}
Also,
\begin{align*}
\vert \dot{c}_{1} - \dot{c}_{2} \vert &\leq \int_{0}^{t} \vert \dot{c}_{1}(s) - \dot{c}_{2}(s) \vert ds\\
&\leq \int_{0}^{t} \left ( \| v_{1} - v_{2} \|_{H^{1}_{x}} \| w_{1} + w_{2} \|_{H^{1}_{x}} + \| v_{1} + v_{2} \|_{H^{1}_{x}} \| w_{1} - w_{2} \|_{H^{1}_{x}} \right ) ds\\
&\leq \delta^{1/2} \left ( \| v_{1} - v_{2} \|_{X^{1,1/2,1}_{\delta}} \| w_{1} + w_{2} \|_{X^{1,1/2,1}} + \| v_{1} + v_{2} \|_{X^{1,1/2,1}_{\delta}} \| w_{1} - w_{2} \|_{X^{1,1/2,1}_{\delta}} \right )\\
&\leq 4cr \delta^{1/2} \left ( \| v_{1} - v_{2} \|_{X^{1,1/2,1}_{\delta}} + \| w_{1} - w_{2} \|_{X^{1,1/2,1}_{\delta}} \right ).
\end{align*}
We thus have
\begin{align*}
&\| \Phi(v_{1}) - \Phi(v_{2}) \|_{X^{s,1/2,1}_{\delta}}\\
\leq  &\left \| \rho(t) \int_{0}^{t} W_{1}(t-s) \big ( (\dot{\gamma}_{1} \partial_{x} + \dot{c}_{1} \partial_{c_{1}}) \rho(t) \psi_{c_{1}(t)} - (\dot{\gamma}_{2} \partial_{x} + \dot{c}_{2} \partial_{c_{2}}) \rho(t) \psi_{c_{2}(t)} \big ) ds \right \|_{X^{s,1/2,1}_{\delta}}\\
\ \ + & \left \| \rho(t) \int_{0}^{t} W_{1}(t-s) \partial_{x}\big ( \rho^{2}( \psi_{c_{0}} v_{1} - \psi_{c_{0}} v_{2}) \big ) ds \right \|_{X^{s,1/2,1}_{\delta}}\\
\ \ + & \left \| \rho(t) \int_{0}^{t} W_{1}(t-s) \partial_{x}\big ( \rho^{2} (v_{1}^{2} - v_{2}^{2}) \big ) ds \right \|_{X^{s,1/2,1}_{\delta}}\\
\leq &c \delta^{\epsilon} \left ( 4c^{2}r \| v_{1} - v_{2} \|_{X^{1,1/2,1}_{\delta}} + 4c^{2}r \| w_{1} - w_{2} \|_{X^{1,1/2,1}_{\delta}} \right ) + c \delta^{\epsilon} \| v_{1} - v_{2} \|_{X^{s,1/2,1}_{\delta}}\\
\ \  + &c \| v_{1} + v_{2} \|_{X^{s,1/2,1}_{\delta}} \| v_{1} - v_{2} \|_{X^{s,1/2,1}_{\delta}}\\
\leq & \delta^{\epsilon} (4c^{3}r + c ) \| v_{1} - v_{2} \|_{X^{s,1/2,1}_{\delta}} + 4c^{2}r \| v_{1} - v_{2} \|_{X^{s,1/2,1}_{\delta}} + 4c^{3}r \delta^{\epsilon} \| w_{1} - w_{2} \|_{X^{s,1/2,1}_{\delta}}.
\end{align*}
It follows that if $\delta,r$ are chosen sufficiently small, then
\begin{equation*}
\| \Phi(v_{1}) - \Phi(v_{2}) \|_{X^{s,1/2,1}_{\delta}} \leq \frac{1}{2} \left ( \| v_{1} - v_{2} \|_{X^{s,1/2,1}_{\delta}} + \| w_{1} - w_{2} \|_{X^{s,1/2,1}_{\delta}} \right ),
\end{equation*}
so that $\Phi$ is a contraction on $\mathcal{B}$.

Turning to estimates for $\Psi$, we find similarly that
\begin{align*}
&\| \Psi(w_{1}) - \Psi(w_{2}) \|_{X^{1,1/2,1}_{\delta}}\\ \leq &(\delta^{\epsilon} + 8c^{3} r^{2} + 8c^{2}r + 4cr) \left ( \| w_{1} - w_{2} \|_{X^{1,1/2,1}_{\delta}} + \| v_{1} - v_{2} \|_{X^{1,1/2,1}_{\delta}} \right ) + 2cr \| w_{1} - w_{2} \|_{X^{1,1/2,1}_{\delta}}\\
+ & 2cr \| w_{1} - w_{2} \|_{X^{1,1/2,1}_{\delta}}.
\end{align*}
We conclude that if $\delta,r$ are sufficiently small, the $\Psi$ is a contraction on $\mathcal{B}$. This establishes the local well-posedness for the weighted and unweighted perturbations.
\end{proof}

\section{Iteration}
In this section we wish to gain long-term control on the behavior of the perturbation by iterating the short-term control gained in Section 3, along with some energy and spectral estimates.  Our goal is to show that $v$ remains bounded in $H^1$ for all time, while $w$ enjoys exponential decay in $H^1$ as time grows.  To do this, we will iterate along local well-posedness time intervals and prove the desired bound by induction.  Specifically, we wish to show that, for all $t > 0$ there exist $c(t)$ and $\gamma(t)$ so that
\begin{enumerate}
\item $c(t)$ and $\gamma(t)$ are smooth functions of time,
\item $\dot{c}$ and $\dot{\gamma}$ are uniformly small, and decay exponentially as $t \to \infty$,
\item $c(t) -c_0$ is uniformly small,
\item $\|v(t)\|_{H^1}$ is uniformly small, and
\item $\|w(t)\|_{H^1}$ decays exponentially as $t \to \infty$.
\end{enumerate}

To reach these conclusions, we rely on the modulation equations described above, \eqref{vequation} and \eqref{wequation}.  The first is a result of a now-standard implicit function theorem argument.

We will prove the rest together via the theorem below, which provides an explicit expression for the decay of $\|w\|_{H^1}$ as a function of time, thereby concluding the exponential decay of the perturbation and the asymptotic stability of the weighted perturbation which are our main result.

\begin{theorem} There is an $\epsilon > 0$ so that if $\|w(0)\|_{H^1}+\|v(0)\|_{H^1}+|c(0)-c_0| <\epsilon$ and $v$, $w$, $c,$ and $\gamma$ are as defined above, then there exist $\kappa$ with $0 < \kappa < 1$ and $C_1 >0$ so that, for any $n \in \N$,
\begin{align}
\|w(n\delta)\|_{H^1} & < \kappa^n\epsilon \notag \\
\|v(n\delta)\|_{H^1} & < C_1\epsilon \notag \\
|\dot{c}(n\delta)| & < \kappa^n\epsilon \label{IH} \\
|\dot{\gamma}(n\delta)| & < \kappa^n\epsilon \notag  \\
|c(n\delta)-c_0| & < (2-\kappa^{n-1})\epsilon. \notag
\end{align}
Here, $\delta$ is the local well-posedness time interval found in Proposition \ref{vwlwp} corresponding to an initial condition of size up to $(2+C_1)\epsilon$, and $ C_1$ depends only on $c_0$.
\end{theorem}

\begin{proof}  First, let $\epsilon$ be sufficiently small so that, whenever $$\|w(t_0)\|_{H^1} + \|v(t_0)\|_{H^1} + |c(t_0)-c_0| <(2+C_1)\epsilon,$$ it follows that  $v(t)$ exists on $[t_0,t_0+\delta]$, and $\|w\|_{X^{1,b}_{[t_0,t_0+\delta]}} + \|v\|_{X^{1,b}_{[t_0,t_0+\delta]}} < C_0(4+2C_1) \epsilon$, where $C_0$ is the implicit constant in the conclusion of Proposition \ref{vwlwp}.

We wish to prove the claim by induction.  First note that $\dot{c}$ and $\dot{\gamma}$ satisfy the following modulation equations:
$$\begin{bmatrix} \dot{\gamma} \\ \dot{c} \end{bmatrix} = \mathcal{A}\left (\begin{bmatrix} \langle e^{ay}\p_y(c-c_0 + v)e^{-ay}w, \eta_1 \rangle_{L^2}\\ e^{ay}\p_y(c-c_0 + v)e^{-ay}w, \eta_2 \rangle_{L^2}\end{bmatrix}\right ),$$ where $$\mathcal{A} = \left ( \begin{bmatrix} 1 + \langle e^{ay}(\p_y\psi_c-\p_y\psi_{c_0}),\eta_1 \rangle -\langle w, \p_y\eta_1 \rangle & \langle e^{ay}(\p_c\psi_c-\p_c\psi_{c_0}),\eta_1 \rangle \\ \langle  e^{ay}(\p_y\psi_c-\p_y\psi_{c_0}),\eta_2 \rangle-\langle w, \p_y\eta_2 \rangle & 1 + \langle e^{ay}(\p_c\psi_c-\p_c\psi_{c_0}),\eta_2 \rangle \end{bmatrix} \right )^{-1}.$$
At any time when $|c-c_0|$ and $\|w\|_{H^1}$ are sufficiently small, it follows that $\|\mathcal{A}\|\leq 2$, so that
$$\left |\begin{bmatrix} \dot{\gamma} \\ \dot{c} \end{bmatrix}\right | \leq 2 \left |\begin{bmatrix} \langle e^{ay}\p_y(c-c_0 + v)e^{-ay}w, \eta_1 \rangle_{L^2}\\ e^{ay}\p_y(c-c_0 + v)e^{-ay}w, \eta_2 \rangle_{L^2}\end{bmatrix}\right |\leq 2(\max_{i=1,2}\|\eta_i\|_{H^1}) (|c-c_0|+\|v\|_{H^1})\|w\|_{H^1}.$$  Therefore \eqref{IH} is satisfied at $t=0$ because of our assumptions on the initial data, so long as $4(\max_{i=1,2}\|\eta_i\|_{H^1})\epsilon \leq 1$.

Now, assume that \eqref{IH} is satisfied at $t=(n-1)\delta$.  We need to control all $5$ quantities going forward to $t=n\delta$.  Without loss of generality assume $\delta \leq 1$.

Let $\eta$ be a sufficiently small constant satisfying \begin{equation}\label{eta} 0< (20+4C_1)\epsilon< \eta \ll 1.\end{equation}   For convenience, define $$L(t)= \|w(t)\|_{H^1} + \|v(t)\|_{H^1} + |\dot{c}(t)| + |\dot{\gamma}(t)| +
|c(t)-c_0|.$$ Note that by the inductive hypothesis \eqref{IH}, $L((n-1)\delta) < (5+C_1)\epsilon<\eta$.  By continuity, then, there exists $\delta_0 > 0 $ so that $L(t)\leq \eta$ on $[(n-1)\delta,(n-1)\delta+\delta_0]$.  Let $\delta_1$ be the largest such $\delta_0$ which is at most $\delta$.  We will first show that $\delta_1= \delta$.

Let us first estimate $\dot{c}$ and $\dot{\gamma}$ on $I:=[(n-1)\delta,(n-1)\delta+\delta_1]$.  As above, we have that for each $t \in I$, $|\dot{c}|+|\dot{\gamma}|\leq C (|c(t)-c_0|+\|v(t)\|_{H^1})\|w(t)\|_{H^1} \leq C \eta^2,$ which is less than $\frac1{10}\eta$ so long as $\eta$ is sufficiently small.  Then $$|c(t)-c_0| \leq |c((n-1)\delta)-c(0)| + \int_I |\dot{c}(t)|dt  \leq (2-\kappa^{n-2})\epsilon + \frac1{10}\eta\delta_1 \leq \frac15\eta.$$ Next, we estimate $\|v(t)\|_{H^1}$.  This can be done using the Lyapunov functional \\${\mathcal{E}}[u]=\int_{-\infty}^\infty \frac12(\p_xu)^2-\frac13u^3+\frac12c_0u^2$ and considering ${\mathcal{E}}[u(t)]-{\mathcal{E}}[u_{c_0}],$ which is a constant of the evolution.  Exactly as in \cite{MR1289328}, this leads to the conclusion  that, for $\eta$ sufficiently small, $\|v(t)\|_{H^1} < C_1\epsilon<\frac14\eta$ on $I$ for some $C_1$ depending only on $c_0$.

Finally, we estimate $\|w(t)\|_{H^1}$.  Define $M= \|w((n-1)\delta)\|_{H^1}^2$, and $N=\|w((n-1)\delta+\delta_1)\|_{H^1}^2$.  Then we have that
\begin{align*}
N-M & = \int_I\frac{d}{dt}\|w(t)\|_{H^1}^2 dt\\
& = 2 \int_I\langle w, w_t \rangle _{H^1} dt \\
& = 2 \int_I\langle w, A_a w + Q\mathcal{F} \rangle_{H^1}dt \\
& = 2 \int_I \langle w, A_a w \rangle_{H^1} dt +2 \int_I \langle w, Q[e^{ay}(\dot{c}\p_c + \dot{\gamma}\p_y)]\psi_c \rangle_{H^1}dt\\ & \phantom{boogah} -2a(\dot{\gamma} + (c-c_0))\int_I \langle w, Q w \rangle_{H^1}dt + 2 \int_I \langle w, Qe^{ay}\p_y(v^2) \rangle_{H^1} dt \\
& = 2 \int_I\langle w, A_a Q w \rangle_{H^1} dt +2 \int_I \langle w, Q[e^{ay}(\dot{c}\p_c + \dot{\gamma}\p_y)]\psi_c \rangle_{H^1}dt\\ & \phantom{boogah} -2a\int_I(\dot{\gamma} + (c-c_0)) \langle w, Q(\p_y-a) w \rangle_{H^1}dt + 2 \int_I \langle w, Qe^{ay}\p_y(v^2) \rangle_{H^1} dt \\
& = \mbox{ (I) + (II) + (III) + (IV) }
\end{align*}

We may conclude by Proposition \ref{H1spec} that (I) is less than or equal to $-2b\int_I \|w(t)\|_{H^1}^2 dt$.

 For (II), we have (II)$=\int_I\langle w,Q [e^{ay}(\dot{c}\p_c + \dot{\gamma}\p_y)](\psi_c-u_{c_0}+u_{c_0}) ]\rangle_{H^1}$.  Since \\ $Q(e^{ay}\p_cu_{c_0})=Q(e^{ay}\p_y(u_{c_0}))=0$, it follows that $$\mathrm{(II)}=\int_I\langle w,Q [e^{ay}(\dot{c}\p_c + \dot{\gamma}\p_y)](\psi_c-u_{c_0}) ]\rangle_{H^1} \lesssim \int_I[ |\dot{c} + \dot{\gamma}||c-c_0|\|w(t)\|_{H^1}].$$

For (III), consider
\begin{align*}
\langle w, Q(\p_y-a)w \rangle_{H^1} & = \langle w, Q\p_yw \rangle_{H^1} -a\langle w, Qw\rangle_{H^1}\\
&   = \langle w, \p_y w \rangle_{H^1} -\langle w, P\p_yw\rangle_{H^1} -a\langle w, Qw \rangle_{H^1} \\
& = 0 - \langle w, \langle \p_y w, \eta_i \rangle_{L^2} \zeta_i \rangle _{H^1} - a\langle w, Qw \rangle_{H^1}\\
& \leq \|\p_yw\|_{L^2} \|\eta_i\|_{L^2} \|w\|_{H^1}\|\xi_i\|_{H^1}  + a \|w\|_{H^1}^2\\
& \lesssim \|w\|_{H^1}^2.
\end{align*}
Therefore, (III) $\lesssim \int_I (|c-c_0| + |\dot{\gamma}|)\|w(t)\|_{H^1}^2dt$.

Finally, we need to estimate (IV).  To do so, we write
$$\langle w, Qe^{ay}\p_y(v^2) \rangle_{H^1} = \langle w, e^{ay}\p_y(v^2)\rangle_{H^1} - \langle w, Pe^{ay}\p_y(v^2)\rangle_{H^1}.$$
Note that $$P(e^{ay}\p)y(v^2) = \sum_{i=1}^2\langle e^{ay} \p_y(v^2),\eta_i\rangle_{L^2}\xi_i,$$ so
\begin{align*}
\int_I\langle w, P(e^{ay}\p_y(v^2))\rangle_{H^1}dt & = \sum_{i=1}^2 \int_I\langle w, \xi_i\rangle_{H^1}\langle wv_y,\eta_i\rangle_{L^2} dt \\ & \lesssim \sum_{i=1}^2 \int_I[\|w\|_{H^1_x}\|\xi_i\|_{H^1_x}\|w\|_{L^2_x}\|v_y\|_{L^2_x}\|\eta_i\|_{L^\infty_x}]dt \\ & \lesssim \int_I(\|w(t)\|_{H^1}^2dt) \|v\|_{L^\infty_tH^1_x}
\end{align*}
Then we need to estimate $$\int_I\langle w, e^{ay}\p_y(v^2)\rangle_{H^1}dt.$$ This has two terms:
$$\int_I\int_\R w_y\p_y(wv_y)dx dt+ \int_I\int_\R w^2v_ydx dt.$$
We estimate the first term as follows:
\begin{align*}
\int_I\int_\R w_y\p_y(wv_y)dx dt & \lesssim \|w_y\|_{X^{0,\frac{1}{2}}}\|\p_y(wv_y)\|_{X^{0,-\frac{1}{2}}} \\
&\lesssim \| w_{y} \|_{X^{0,1/2,1}} \| \p_y (w v_y) \|_{X^{0,1/2,1}}\\
& \lesssim  \|w\|^2_{X^{1,\frac12,1}}\|v\|_{X^{1,\frac{1}{2},1}}\\
& \lesssim M \|v\|_{X^{1,\frac{1}{2},1}},
\end{align*}
using \eqref{desiredbilinear}. For the second term, we get:
\begin{align*}
| \int_n^{n+1}\int_\R w^2v_ydx dt| & \lesssim \|w\|_{L^6_tL^\infty_x}\|wv_y\|_{L^\frac65_tL^1_x}\\
& \lesssim \|w\|_{L^6_tL^\infty_x}\|w\|_{L^\frac{12}5_tL^2_x}\|v_y\|_{L^\frac{12}5_tL^2_x}\\
& \lesssim  \|w\|_{L^{\infty}_{t}L^\infty_x}\|w\|_{L^\infty_tL^2_x}\|v_y\|_{L^\infty_tL^2_x}\\
& \lesssim  \|w\|_{X^{1,\frac{1}{2},1}}^2\|v_y\|_{X^{1,\frac{1}{2},1}}\\
& \lesssim M\|v\|_{X^{1,\frac{1}{2},1}}
\end{align*}
via Strichartz estimates.

In total, we obtain the following estimate for the increment of $w$:
\begin{align*}
N-M &\leq \int_I[-2b+ C|\dot{\gamma} + c-c_0|+C\|v\|_{L^\infty_tH^1_x}]\|w(t)\|_{H^1_x}^2dt\\
&\quad + \int_I C[ |\dot{c} + \dot{\gamma}||c-c_0|\|w(t)\|_{H^1}]dt+M\|v\|_{X^{1,\frac12}}.
\end{align*}
Using our controls above, this yields
$$N-M \leq (-2b+C\eta)\eta^2\delta+C\eta^2\eta\delta+ C_1M\eta$$
Hence we may conclude that $$N\leq M(1+C_1\eta)+C\eta^3.$$
Therefore, it follows that $N\leq \epsilon^2(1+C_1\eta) + C \eta^3$, so, for $\eta$ sufficiently small,  $\|w((n-1)\delta+\delta_1)\|_{H^1} \leq \frac14\eta$.

Finally, we conclude that $L((n-1)\delta+\delta_1) \leq (\frac{1}{10}+\frac{1}{5}+\frac14+\frac14)\eta< \eta$.  Hence, we may continue past $\delta_1$ with \eqref{eta} remaining valid.  Hence $\delta_1=\delta$.  Therefore, we have that $L(t) \leq \eta$ on $[(n-1)\delta,n\delta]$.  Now, let us verify \eqref{IH} at $t=n\delta$. Set $J:=[(n-1)\delta,n\delta]$.

As above, we have that for each $t \in I$, $$|\dot{c}|+|\dot{\gamma}|\leq C (|c(t)-c_0|+\|v(t)\|_{bH^1})\|w(t)\|_{H^1} \leq C (1+C_1)\eta \|w(t)\|_{H^1},$$ which is less than $\|w(t)\|_{H^1}$ so long as $\eta$ is sufficiently small.  Hence, the control on $\dot{c}$ and $\dot{\gamma}$ is valid whenever the control on $w$ holds.  When it holds, then, $$|c(t)-c_0| \leq |c((n-1)\delta)-c(0)| + \int_J |\dot{c}(t)|dt  \leq (2-\kappa^{n-2})\epsilon + \kappa^{n-1} \leq (2-\kappa^{n-1}) \epsilon.$$  The estimate on $v$ is the same as above, with the same result.

Finally, we estimate $\|w(t)\|_{H^1}$.  Define $m(n)=\inf_{J}\|w(t)\|_{H^1}^2$ and $N(n)=\|w(n\delta)\|_{H^1}^2$.  Then we have that
$$
N(n)-N(n-1) = \int_J\frac{d}{dt}\|w(t)\|_{H^1}^2 dt,
$$
which has the same four terms to be estimated as above.  As above, we obtain the following estimate for the increment of $w$:
\begin{equation*}
\begin{aligned}
N(n)-N(n-1) &\leq  \int_J[-2b+ C|\dot{\gamma} + c-c_0|+C\|v\|_{L^\infty_tH^1_x}]\|w(t)\|_{H^1_x}^2dt\\
 &\quad + \int_J C[ |\dot{c} + \dot{\gamma}||c-c_0|\|w(t)\|_{H^1}]dt+N(n-1)\|v\|_{X^{1,\frac12,1}}.
\end{aligned}
\end{equation*}
Using our controls above, this yields
$$N(n)-N(n-1) \leq \int_I[-2b+C\eta]\|w(t)\|_{H^1_x}^2dt + 2N(n-1)C_1\epsilon,$$
So, for $\eta$ sufficiently small, we have
$$N(n)-N(n-1) \leq -bm(n)+ C\epsilon N(n-1).$$

In order to close the loop, we need to relate $m(n)$ and $N(n-1)$.  There are two possible cases.  First, suppose that $m \geq \frac34 N(n)$.  Then in the above argument we obtain
\begin{equation}\label{largemincrement}N(n)-N(n-1) \leq -\frac34bN(n-1)+ C\epsilon N(n-1)
\end{equation}
On the other hand, if $m(n) < \frac34 N(n-1)$, then $|m(n)-N(n)| > \frac14 N(n-1).$  Let $t^*$ be the time at which the minimum value $m(n)$ occurs.  By the increment calculation above, then, we have that
\begin{align*}
\frac14N(n-1) & < N(n-1)-m(n) \\
& = |\int_{(n-1)\delta}^{t^*}\langle w, A_aw + Q\mathcal{F} \rangle_{H^1} dt|.
\end{align*}
Therefore,
$$|\int_{(n-1)\delta}^{t^*} \langle w, A_aw\rangle{H^1} dt| \geq \frac14N(n-1)-|\int_{(n-1)\delta}^{t^*} \langle w, Q\mathcal{F}\rangle_{H^1} dt|.$$
By the increment calculation above, we then obtain
$$|\int_{n\delta}^{t^*} \langle w, A_aw\rangle{H^1} dt| \geq \frac14N(n-1)-(C\eta N(n-1)+C\epsilon N(n-1)).$$ Since $\langle w, A_aw\rangle_{H^1} \leq 0$ for all $w$ by Proposition \ref{H1spec}, i.e. this quantity has a definite sign, it follows that
$$|\int_J\langle w, A_aw\rangle_{H^1}dt| \geq \frac14N(n-1)-(C(\eta+\epsilon) N(n-1)).$$
Hence in this case,
\begin{equation}\label{smallmincrement}N(n)-N(n-1) \leq -\frac14N(n-1)+2(C(\eta+\epsilon) N(n-1)).
\end{equation}
Hence, in either case, it follows that, with $\beta=\min(\frac34b,\frac1{4})$,
\begin{equation}\label{incrementconclusion}
N(n)-N(n-1) \leq -\beta N(n-1)+C\eta N(n-1),
\end{equation}
So, $N(n) \leq (1+C\eta-\beta)N(n-1)$.  For $\eta$ sufficiently small, it follows that with $\kappa :=1-\frac{\beta}{2}$, $\kappa <1$ and
$N(n) \leq \kappa N(n-1)$.  So, since $N(n-1) \leq \kappa^{n-1}\epsilon,$ $N(n) \leq \kappa^n\epsilon$.  By the arguments above, the corresponding controls on $\dot{c}$, $\dot{\gamma}$, $c-c_0$, and $\|v(n)\|_{H^1}$ immediately follow.  Hence, by induction, the theorem holds.
\end{proof}

% The bibliography.
\bibliographystyle{amsplain} % (change according to your preference)
%% ***   Set the bibliography file.   ***
\bibliography{H1references}

\end{document}